\documentclass[smallextended]{svjour3}       
\smartqed
\usepackage{amsmath,mathrsfs,multirow,amssymb}
\usepackage{enumerate}
\usepackage{color,url}
\usepackage{indentfirst}
\usepackage{graphicx}
\usepackage{fancyhdr}
\usepackage[ruled,noline,linesnumbered]{algorithm2e}
\usepackage{subfigure}
\usepackage{rotating}
\usepackage{tabularx}
\usepackage{epstopdf}
\usepackage{caption}

\begin{document}

\title{On the acceleration of the Barzilai-Borwein method}



\author{Yakui Huang\and
        Yu-Hong Dai\and
        Xin-Wei Liu\and
        Hongchao Zhang
}


\institute{
Yakui Huang \at
             Institute of Mathematics, Hebei University of Technology, Tianjin 300401, China\\
              \email{huangyakui2006@gmail.com}
              \and
Yu-Hong Dai \at
              LSEC, ICMSEC, Academy of Mathematics and Systems Science, Chinese Academy of Sciences, 100190, Beijing, China\\
              \email{dyh@lsec.cc.ac.cn}           
           \and
              Xin-Wei Liu \at
              Institute of Mathematics, Hebei University of Technology, Tianjin 300401, China\\
              \email{mathlxw@hebut.edu.cn}
              \and
              Hongchao Zhang \at
              Department of Mathematics, Louisiana State University, Baton Rouge, LA 70803-4918, USA\\
               \email{hozhang@math.lsu.edu}
}
\date{Received: date / Accepted: date}

\maketitle

\begin{abstract}
The Barzilai-Borwein (BB) gradient method is efficient for solving large-scale
unconstrained problems to the modest accuracy and has a great advantage of being easily
 extended to solve a wide class of constrained optimization problems.
In this paper, we propose a new stepsize to accelerate the BB method
by requiring finite termination for minimizing two-dimensional strongly convex quadratic function.
Combing with this new stepsize, we develop gradient methods which adaptively take
the nonmonotone BB stepsizes and certain monotone stepsizes
for minimizing general strongly convex quadratic function.
Furthermore, by incorporating nonmonotone line searches and gradient projection techniques,
we extend these new gradient methods to solve general smooth unconstrained
and bound constrained optimization.
Extensive numerical experiments show that our strategies of properly inserting monotone
gradient steps into the nonmonotone BB method could significantly improve its performance
and the new resulted methods can outperform the most successful gradient decent
methods developed in the recent literature.
\end{abstract}

\section{Introduction}
\label{intro}
Gradient descent methods have been widely used for solving smooth unconstrained nonlinear optimization
\begin{equation}\label{eqpro}
  \min_{x\in\mathbb{R}^n}~f(x)
\end{equation}
by generating a sequence of iterates
\begin{equation}\label{eqitr}
  x_{k+1}=x_k-\alpha_kg_k,
\end{equation}
where $f: \mathbb{R}^n \to \mathbb{R}$ is continuously differentiable,
$g_k=\nabla f(x_k)$ and $\alpha_k>0$ is the stepsize along the negative gradient.
Different gradient descent methods would have different rules for determining the stepsize $\alpha_k$.
The classic steepest descent (SD) method proposed by Cauchy \cite{cauchy1847methode}
determines its stepsize by the so-called exact line search
\begin{equation}\label{sd}
  \alpha_k^{SD}=\arg\min_{\alpha \in \mathbb{R}}~f(x_k-\alpha g_k).
\end{equation}
Although the SD method locally has the most function value reduction along the negative gradient direction,
it often performs poorly in practice. Theoretically, when $f$ is a strongly convex quadratic function,
i.e., \begin{equation}\label{qudpro}
  f(x)=\frac{1}{2}x ^T Ax-b ^T x,
\end{equation}
where $b\in\mathbb{R}^n$ and $A\in\mathbb{R}^{n\times n}$ is symmetric and positive definite,
SD method converges $Q$-linearly \cite{akaike1959successive} and will asymptotically perform
zigzag between two orthogonal directions \cite{forsythe1968asymptotic,huang2019asymptotic}.

In 1988, Barzilai and Borwein \cite{barzilai1988two} proposed the following two novel stepsizes that
significantly improve the performance of gradient descent methods:
\begin{equation}\label{sBB}
  \alpha_k^{BB1}=\frac{s_{k-1}^Ts_{k-1}}{s_{k-1}^Ty_{k-1}} \qquad \mbox{and} \qquad
  \alpha_k^{BB2}=\frac{s_{k-1}^Ty_{k-1}}{y_{k-1}^Ty_{k-1}},
\end{equation}
where $s_{k-1}=x_k-x_{k-1}$ and $y_{k-1}=g_k-g_{k-1}$.
Clearly, when $s_{k-1}^Ty_{k-1}>0$, one has $\alpha_{k}^{BB1}\geq\alpha_{k}^{BB2}$.
Hence,  $\alpha_{k}^{BB1}$ is often called the \emph{long} BB stepsize
while $\alpha_{k}^{BB2}$ is called the \emph{short} BB stepsize.
When the objective function is quadratic \eqref{qudpro},
the BB stepsize $\alpha_{k}^{BB1}$ will be exactly the steepest descent stepsize, but with
one step retard,  while $\alpha_{k}^{BB2}$ will be just the stepsize of minimal gradient (MG) method \cite{dai2003altermin},
that is
\[
  \alpha_k^{BB1}=\frac{g_{k-1}^Tg_{k-1}}{g_{k-1}^TAg_{k-1}}=\alpha_{k-1}^{SD}
  \quad \mbox{and} \quad \alpha_k^{BB2}=\frac{g_{k-1}^TAg_{k-1}}{g_{k-1}^TA^2g_{k-1}}=\alpha_{k-1}^{MG}.
\]
It is proved that the Barzilai-Borwein (BB) method converges $R$-superlinearly for minimizing two-dimensional strongly convex quadratic function \cite{barzilai1988two}
and $R$-linearly for the general $n$-dimensional case \cite{dai2002r}.
Although the BB method does not decrease the objective function value
monotonically, extensive numerical experiments show that it performs much better than the SD method \cite{fletcher2005barzilai,raydan1997barzilai,yuan2008step}.
And it is commonly accepted that when a not high accuracy is required,
 BB-type methods could be even competitive with nonlinear conjugate gradient (CG) methods for solving smooth unconstrained optimization  \cite{fletcher2005barzilai,raydan1997barzilai}.
Furthermore, by combing with gradient projection techniques, the BB-type methods have
a great advantage of easy extension to solve a wide class of
constrained optimization, for example the bound or simplex constrained optimization \cite{dai2005projected}.
Hence, BB-type methods enjoy many important applications, such as image restoration \cite{wang2007projected}, signal processing \cite{liu2011coordinated}, eigenvalue problems \cite{jiang2013feasible}, nonnegative matrix factorization \cite{huang2015quadratic}, sparse reconstruction \cite{wright2009sparse}, machine learning \cite{tan2016barzilai}, etc.

Recently, Yuan \cite{yuan2006new,yuan2008step} propose a
gradient descent method which combines a new stepsize
\begin{equation}\label{syv}
    \alpha_k^{Y}=\frac{2}{\frac{1}{\alpha_{k-1}^{SD}}+\frac{1}{\alpha_{k}^{SD}}+
  \sqrt{\left(\frac{1}{\alpha_{k-1}^{SD}}-\frac{1}{\alpha_{k}^{SD}}\right)^2+
  \frac{4\|g_k\|^2}{(\alpha_{k-1}^{SD}\|g_{k-1}\|)^2}}},
\end{equation}
in the SD method so that the new method enjoys finite termination  for minimizing
a two-dimensional strongly convex quadratic function.
Based on this new stepsize $\alpha_k^Y$,
Dai and Yuan \cite{dai2005analysis} further develop the DY method,
which alternately employs $\alpha_k^{SD}$ and $\alpha_k^{Y}$ stepsizes as follows
\begin{equation}\label{sdy}
  \alpha_k^{DY}=\left\{
                  \begin{array}{ll}
                    \alpha_k^{SD}, & \hbox{if mod($k$,4)$<2$;} \\
                    \alpha_k^{Y}, & \hbox{otherwise.}
                  \end{array}
                \right.
\end{equation}
It is easy to see that $\alpha_k^{Y}\leq\alpha_k^{SD}$. Hence, DY method \eqref{sdy} is a monotone method.
 Moreover, it is shown that DY method \eqref{sdy} performs better than the nonmonotone BB method \cite{dai2005analysis}.

The property of nonmonotonically reducing objective function values is an intrinsic feature that
causes the efficiency of BB method.  However, it is also pointed out by Fletcher \cite{fletcher2012limited}
 that retaining monotonicity is important for a gradient method, especially for minimizing general
 objective functions.
 On the other hand, although the monotone DY method performs well, using $\alpha_k^{SD}$ and $\alpha_k^{Y}$
  in a nonmonotone fashion may yield better performance, see \cite{de2014efficient} for example.
Moreover, it is usually difficult to compute the exact monotone stepsize $\alpha_k^{SD}$
in general optimization.
Hence, in this paper,  motivated by the great success of the BB method and the previous considerations,
we want to further improve and accelerate the \emph{nonmonotone} BB method by incorporating
some \emph{monotone} steps.
For a more general and uniform analysis, we first consider to accelerate the class of
 gradient descent methods \eqref{eqitr} for quadratic optimization  \eqref{qudpro}
 using the following stepsize
\begin{equation}\label{glstep1}
  \alpha_k(\Psi(A))=\frac{g_{k-1}^T\Psi(A) g_{k-1}}{g_{k-1}^T\Psi(A)Ag_{k-1}},
\end{equation}
where $\Psi(\cdot)$ is a real analytic function on $[\lambda_1,\lambda_n]$ that can be expressed by a Laurent series
\begin{equation*}
  \Psi(z)=\sum_{k=-\infty}^\infty c_kz^k,~~c_k\in\mathbb{R},
\end{equation*}
such that $0<\sum_{k=-\infty}^\infty c_kz^k<+\infty$ for all $z\in[\lambda_1,\lambda_n]$. Here, $\lambda_1$ and $\lambda_n$ are the smallest and largest eigenvalues of $A$, respectively. Clearly, the method \eqref{glstep1} is generally nonmonotone and the two BB stepsizes $\alpha_k^{BB1}$ and $\alpha_k^{BB2}$ can be obtained by setting $\Psi(A)=I$ and $\Psi(A)=A$ in \eqref{glstep1}, respectively.

More particularly, we will derive a new stepsize, say $\tilde{\alpha}_k(\Psi(A))$,
which together with the stepsize $\alpha_k(\Psi(A))$ in \eqref{glstep1} can
minimize the two-dimensional convex quadratic function in no more than $5$ iterations.
To the best of our knowledge, this is the first  nonmonotone gradient method with finite termination property.
We will see that $\tilde{\alpha}_k(I)\leq\alpha_k^{SD}$ and $\tilde{\alpha}_k(A)\leq\alpha_k^{MG}$.
Hence, this finite termination property is essentially obtained by inserting monotone stepsizes into
 the generally nonmonotone stepsizes \eqref{glstep1}.
In fact, we show that this finite termination property can be maintained even the algorithm
uses different function $\Psi$'s during its iteration.  Based on this observation,
 to achieve good numerical performance, we propose an adaptive nonmonotone gradient method
(ANGM), which  adaptively takes some nonmonotone steps involving the long and short BB stepsizes \eqref{sBB},
and some monotone steps using $\tilde{\alpha}_k(A)$.
Moreover, to efficiently minimize more general nonlinear objective function,
we propose two variants of ANGM, called ANGR1 and ANGR2, using certain retard stepsize.
By combing nonmonotone line search and gradient projection techniques,
these two variants of gradient methods are further extended to solve bound constrained optimization.
Our numerical experiments show that the new proposed methods significantly accelerate
the BB method and perform much better on minimizing quadratic function \eqref{qudpro}
than the most successful gradient decent methods developed in the recent literature, such as
DY method \cite{dai2005analysis}, ABBmin2 method \cite{frassoldati2008new} and SDC method
\cite{de2014efficient}. In addition, we also compare
ANGR1 and ANGR2 with the spectral projected gradient (SPG) method
\cite{birgin2000nonmonotone,birgin2014spectral} and the BB method
using Dai-Zhang nonmonotone line search (BB1-DZ) \cite{dai2001adaptive} for solving
the general unconstrained problems from \cite{andrei2008unconstrained}
and the bound constrained problems from the CUTEst collection \cite{gould2015cutest}.
The numerical results highly suggest the potential benefits of our new proposed methods
for solving more general unconstrained and bound constrained optimization.

The paper is organized as follows. In Section \ref{secabb}, we derive the new stepsize
$\tilde{\alpha}_k(\Psi(A))$ by requiring finite termination on minimizing two-dimensional strongly
convex quadratic function.
In Section \ref{secarbbmd}, we first derive the ANGM, ANGR1 and ANGR2 methods for minimizing
general strongly convex quadratic function and then generalize ANGR1 and ANGR2 to solve bound
 constrained optimization.
Our extensive numerical experiments on minimizing strongly convex quadratic function and
solving general unconstrained and bound constrained optimization are presented in Section \ref{secnum}.
We finally draw some conclusion remarks in  Section \ref{secclu}.

\vspace{-2mm}
\section{Derivation of new stepsize}\label{secabb}
In this section, we derive a new monotone stepsize based on the nonmonotone gradient method \eqref{glstep1}
to minimize quadratic function \eqref{qudpro}.
This new stepsize is motivated by requiring finite termination for minimizing two-dimensional
strongly convex quadratic function. Such an idea was originally proposed by Yuan \cite{yuan2006new}
to accelerate SD method. However, new techniques need to be developed for
accelerating the class of nonmonotone gradient descent methods \eqref{glstep1} since the key
orthogonal property of successive two gradients generated by SD method no longer holds for methods \eqref{glstep1}.

Observe that the method \eqref{glstep1} is invariant under translations and rotations
when minimizing quadratics. Hence, for theoretical analysis to minimize \eqref{qudpro},
without loss of generality,
we may simply assume that the matrix $A$ is diagonal, i.e.,
\begin{equation}\label{formA}
  A=\textrm{diag}\{\lambda_1,\ldots,\lambda_n\},
\end{equation}
where $0 < \lambda_1\leq\ldots\leq\lambda_n$.

\subsection{Motivation}\label{submot}
First, let us investigate the behavior of gradient method \eqref{glstep1} with $\Psi(A)=I$ (i.e., the BB1 method). Particularly, we apply it to the non-random quadratic minimization problem
proposed in \cite{de2014efficient},
which has the form \eqref{qudpro} with a diagonal matrix $A$ given by
\begin{equation}\label{pro2}
  A_{jj}=10^{\frac{ncond}{n-1}(n-j)}, ~~~j=1,\ldots,n,
\end{equation}
and $b$ being a null vector. Here, $ncond=\log_{10} \kappa$ and $\kappa>0$ is the condition number of $A$.
We set $n=10$, $\kappa = 10^3$ and use $(10,10,\ldots,10)^T$ as the initial point.
The iteration was stopped once the gradient norm is reduced by a factor of $10^{-6}$.
Denote the $i$-th component of $g_k$ by $g_k^{(i)}$ and
the indices of the components of $g_k$ with two largest magnitudes by $i_1$ and $i_2$, respectively.
Then, the percentage of the magnitudes of the first two largest components to that of the whole gradient
can be computed by
\begin{equation*}
  \Upsilon(g_k)=\frac{|g_k^{(i_1)}|+|g_k^{(i_2)}|}{\sum_{i=1}^n|g_k^{(i)}|}.
\end{equation*}
This $\Upsilon(g_k)$ is plotted in Fig. \ref{bbdom2} (left), where we can see that
$\Upsilon(g_k)\geq0.8$ holds for more than half of the iterations (145 out of 224 total iterations).
Hence, roughly speaking, the searches of the BB1 method are often dominated in
some two-dimensional subspaces.
The history of index $i_1$ against the iteration number is also plotted in Fig. \ref{bbdom2} (right),
where we can see that $|g_k^{(i_1)}|$  corresponds more frequently to the largest eigenvalues $\lambda_{10}$ or $\lambda_{9}$.
Since
\begin{equation}\label{updgki}
  g_{k+1}^{(j)}=(1-\alpha_k\lambda_j)g_k^{(j)},~j=1,\ldots,n.
\end{equation}
and $1/\lambda_{n}\leq\alpha_k\leq1/\lambda_1$, the history of $i_1$ in Fig. \ref{bbdom2} (right)
in fact indicates that, most stepsizes generated by the BB1 method are often much larger than
$1/\lambda_{10}$ or $1/\lambda_{9}$.
As a result, the BB1 method may need many iterations to reduce the corresponding components of the gradients
$g_k^{(9)}$ or $g_k^{(10)}$.
\begin{figure}[thp!b]
  \centering
  \subfigure{\includegraphics[width=0.49\textwidth,height=0.4\textwidth]{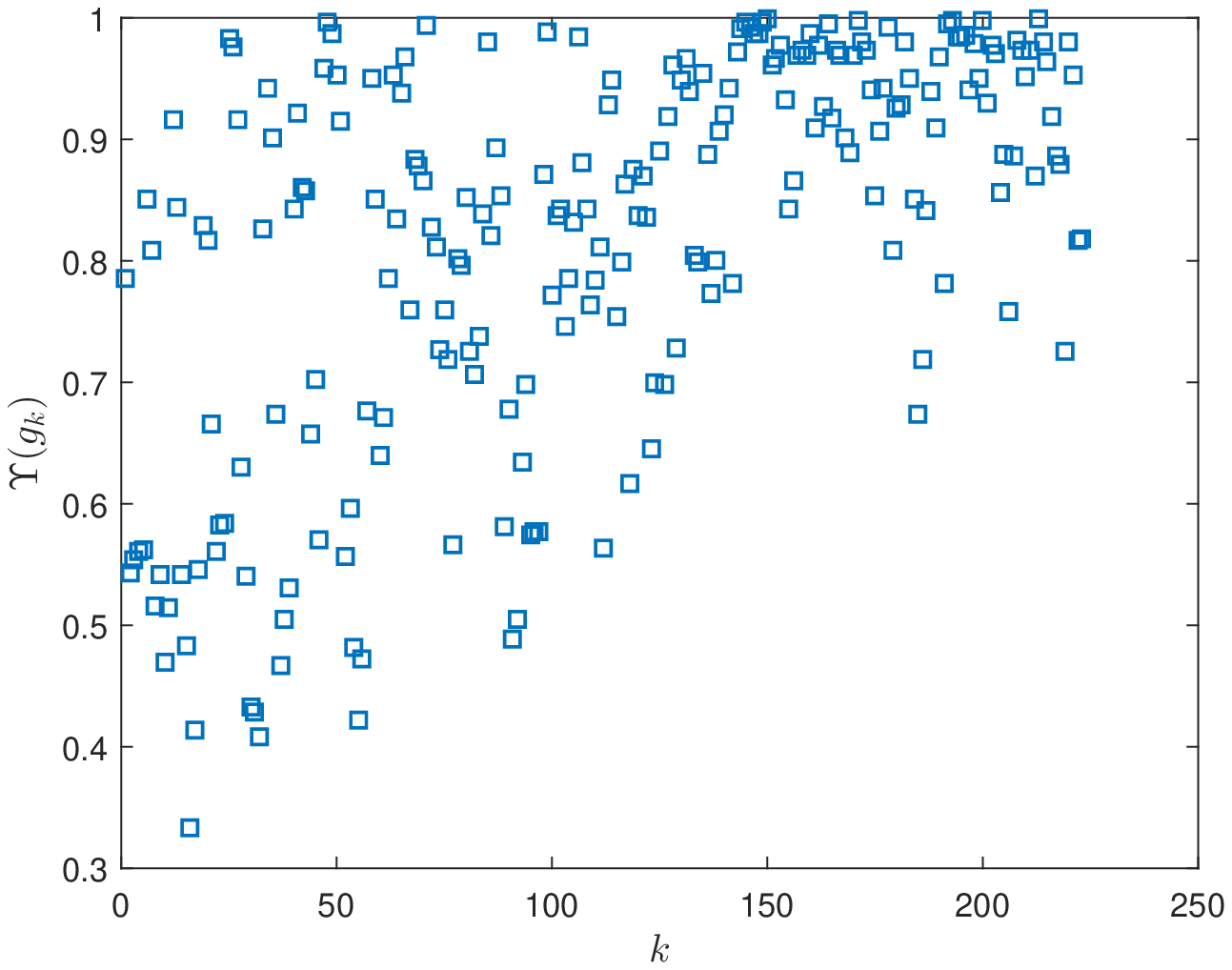}}
  \subfigure{\includegraphics[width=0.49\textwidth,height=0.4\textwidth]{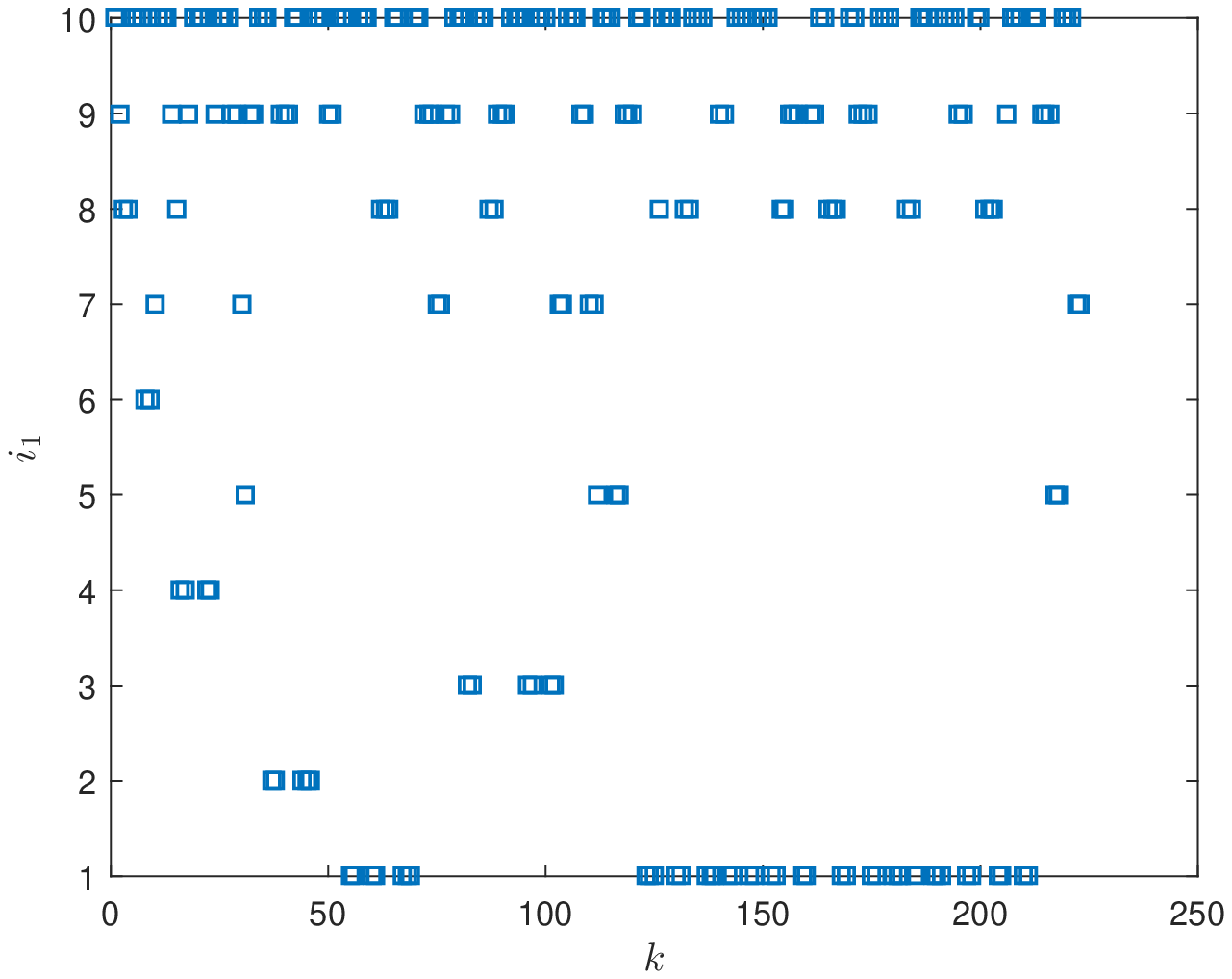}}
  \\
  \caption{Problem \eqref{pro2} with $n=10$: history of $\Upsilon(g_k)$ (left) and the index $i_1$ (right) generated by the BB1 method}\label{bbdom2}
\end{figure}

In \cite{huang2019asymptotic}, we showed that a family of gradient methods including SD and MG will asymptotically reduce
their searches in a two-dimensional subspace and
could be accelerated by exploiting certain orthogonal properties in this two-dimensional subspace.
In a similar spirit, we could also accelerate the convergence of
the class of gradient methods \eqref{glstep1}
in a lower dimensional subspace if certain orthogonal properties hold.

Suppose that, for a given $k>0$, there exists a $q_k$ satisfying
\begin{equation}\label{defqk}
  (I-\alpha_{k-1}A)q_k=g_{k-1}.
\end{equation}
Since this $q_k$ is also invariant under translations and rotations, for later analysis
we may still assume $A$ in \eqref{defqk} is diagonal as in \eqref{formA}.
The following lemma shows a generalized orthogonal property for $q_k$ and $g_{k+1}$,
which is a key property for deriving our new stepsize in the next subsection.

\begin{lemma}[Orthogonal property]\label{orpty}
Suppose that the sequence $\{g_k\}$ is obtained by applying gradient method \eqref{eqitr}
with stepsizes \eqref{glstep1}
to minimize a quadratic function \eqref{qudpro} and $q_k$ satisfies \eqref{defqk}. Then, we have
  \begin{equation}\label{orthg}
  q_k^T\Psi(A)g_{k+1}=0.
\end{equation}
\end{lemma}
\begin{proof}
By \eqref{eqitr}, \eqref{glstep1} and \eqref{defqk} we get
\begin{align*}
  q_k^T\Psi(A)g_{k+1}
 &=q_k^T\Psi(A)(I-\alpha_kA)(I-\alpha_{k-1}A)g_{k-1}
\\
&=g_{k-1}^T\Psi(A)(I-\alpha_kA)g_{k-1}
\\
&=g_{k-1}^T\Psi(A)g_{k-1}-\alpha_kg_{k-1}^T\Psi(A)Ag_{k-1}=0.
\end{align*}
This completes the proof.
\qed
\end{proof}

\subsection{A new stepsize}\label{sectwod}

In this subsection, we derive a new stepsize based on the iterations of gradient method \eqref{glstep1}.
We show that combining the new stepsize with gradient method \eqref{glstep1}, we can
achieve finite termination for minimizing two-dimensional strongly convex quadratic functions.

By Lemma \ref{orpty}, we have that $g_{k}^T\Psi(A)q_{k-1}=0$ for $k>0$.
Now, suppose both $\Psi^r(A)q_{k-1}$ and $\Psi^{1-r}(A)g_{k}$ are nonzero vectors, where $r \in \mathbb{R}$.
Let us consider to minimize the function $f$ in a two-dimensional subspace spanned by
 $\frac{\Psi^r(A)q_{k-1}}{\|\Psi^r(A)q_{k-1}\|}$ and $\frac{\Psi^{1-r}(A)g_{k}}{\|\Psi^{1-r}(A)g_{k}\|}$,
 and let
\begin{align}\label{ftu1}
\varphi(t,l) &:=
 f\left(x_{k}+t\frac{\Psi^r(A)q_{k-1}}{\|\Psi^r(A)q_{k-1}\|}
  +l\frac{\Psi^{1-r}(A)g_{k}}{\|\Psi^{1-r}(A)g_{k}\|}\right) \nonumber\\
& =
f(x_{k}) +\vartheta_k^T
              \begin{pmatrix}
                t \\
                l \\
              \end{pmatrix}
            +  \frac{1}{2}
              \begin{pmatrix}
                t \\
                l \\
              \end{pmatrix} ^T
H_k
              \begin{pmatrix}
                t \\
                l \\
              \end{pmatrix},
\end{align}
where
\begin{equation}\label{eqgrad}
  \vartheta_k= B_k g_{k} =  \begin{pmatrix}
                \frac{g_{k}^T\Psi^r(A)q_{k-1}}{\|\Psi^r(A)q_{k-1}\|}\\
                \frac{g_{k}^T\Psi^{1-r}(A)g_{k}}{\|\Psi^{1-r}(A)g_{k}\|} \\
              \end{pmatrix}
\mbox{ with  }    B_k=\begin{pmatrix}
      \frac{\Psi^r(A)q_{k-1}}{\|\Psi^r(A)q_{k-1}\|},
      \frac{\Psi^{1-r}(A)g_{k}}{\|\Psi^{1-r}(A)g_{k}\|} \\
    \end{pmatrix} ^T
\end{equation}
and
\begin{equation} \label{eqhes}
  H_k=  B_k A B_k^T  =  \begin{pmatrix}
      \frac{q_{k-1}^T\Psi^{2r}(A)Aq_{k-1}}{\|\Psi^r(A)q_{k-1}\|^2} & \frac{q_{k-1}^T\Psi(A) Ag_{k}}{\|\Psi^r(A)q_{k-1}\|\|\Psi^{1-r}(A)g_{k}\|} \\
      \frac{q_{k-1}^T\Psi(A)Ag_{k}}{\|\Psi^r(A)q_{k-1}\|\|\Psi^{1-r}(A)g_{k}\|}  & \frac{g_{k}^T\Psi^{2(1-r)}(A)Ag_{k}}{\|\Psi^{1-r}(A)g_{k}\|^2}\\
    \end{pmatrix}.
\end{equation}
Denote the components of $H_k$ by $H_k^{(ij)}$, $i,j=1,2$ and notice that $B_k B_k^T = I$
by $g_{k}^T\Psi(A)q_{k-1}=0$.
Then, we have the following finite termination theorem.

\begin{theorem}[Finite termination]\label{thftm}
Suppose that a gradient method \eqref{eqitr} is applied to minimize a two-dimensional quadratic function \eqref{qudpro} with $\alpha_k$ given by \eqref{glstep1}
for all $k \neq k_0$ and uses the stepsize
  \begin{align}\label{new2}
  \tilde{\alpha}_{k_0}
  &= \frac{2}{(H_{k_0}^{(11)}+H_{k_0}^{(22)})+\sqrt{(H_{k_0}^{(11)}-H_{k_0}^{(22)})^2+4(H_{k_0}^{(12)})^2}}
\end{align}
 at the $k_0$-th iteration where $k_0 \geq2$.  Then, the method will find the minimizer in at most $k_0+3$ iterations.
\end{theorem}
\begin{proof}
Let us suppose  $x_k$ is not a minimizer for all $k=1, \ldots, k_0+2$. We then show $k_0+3$ must be the minimizer, i.e., $g_{k_0+3} = 0$. For notation convenience, in the following proof of this
theorem, let's simply use $k$ to denote $k_0$.
First, we show that using stepsize \eqref{new2} at the $k$-th iteration implies
\begin{equation} \label{parallel-g2}
g_{k+1}\quad \mbox{is parallel to} \quad - B_k^T  H_k^{-1} \vartheta_k + \tilde{\alpha}_{k} g_{k},
\end{equation}
where $\vartheta_k, B_k$ and $H_k$ is given by \eqref{eqgrad} and \eqref{eqhes}.
In fact, $\tilde{\alpha}_{k}$ given by \eqref{new2} satisfies the following quadratic equation
\begin{equation}\label{eqalp2}
  \tilde{\alpha}_{k}^2\Delta -\tilde{\alpha}_{k}(H_k^{(11)}+H_k^{(22)})+1=0,
\end{equation}
where $\Delta=\textrm{det}(H_k) = \textrm{det}(A)  >0$.
Let
\begin{equation*}
  \Theta=(H_k^{(12)}\vartheta_k^{(1)}+H_k^{(22)}\vartheta_k^{(2)})\vartheta_k^{(1)}
  -(H_k^{(11)}\vartheta_k^{(1)}+H_k^{(12)}\vartheta_k^{(2)})\vartheta_k^{(2)},
\end{equation*}
where $\vartheta_k^{(i)}$ are components of $\vartheta_k$, $i=1,2$.
Then, multiplying $\Theta$ to \eqref{eqalp2}, we have
\begin{equation}\label{eqalp1}
  \tilde{\alpha}_{k}^2\Delta \Theta- \tilde{\alpha}_{k}(H_k^{(11)}+H_k^{(22)})\Theta+\Theta=0,
\end{equation}
which is exactly
\begin{align}\label{parallel-explicit}
  &(H_k^{(22)}\vartheta_k^{(1)}-H_k^{(12)}\vartheta_k^{(2)}-\tilde{\alpha}_k\Delta \vartheta_k^{(1)})
  [\vartheta_k^{(2)}- \tilde{\alpha}_k(H_k^{(12)}\vartheta_k^{(1)}+H_k^{(22)}\vartheta_k^{(2)})] \nonumber\\
  =&(H_k^{(11)}\vartheta_k^{(2)}-H_k^{(12)}\vartheta_k^{(1)}-\tilde{\alpha}_k\Delta \vartheta_k^{(2)})
  [\vartheta_k^{(1)}-\tilde{\alpha}_k(H_k^{(11)}\vartheta_k^{(1)}+H_k^{(12)}\vartheta_k^{(2)})].
\end{align}
The above identity \eqref{parallel-explicit} implies the vector
\begin{equation*}
              \begin{pmatrix}
                \vartheta_k^{(1)}- \tilde{\alpha}_k(H_k^{(11)}\vartheta_k^{(1)}+H_k^{(12)}\vartheta_k^{(2)}) \\
                \vartheta_k^{(2)}- \tilde{\alpha}_k(H_k^{(12)}\vartheta_k^{(1)}+H_k^{(22)}\vartheta_k^{(2)})\\
              \end{pmatrix}
\end{equation*}
is parallel to
\begin{equation*}
              \begin{pmatrix}
           H_k^{(22)}\vartheta_k^{(1)}-H_k^{(12)}\vartheta_k^{(2)}- \tilde{\alpha}_k\Delta \vartheta_k^{(1)} \\
           H_k^{(11)}\vartheta_k^{(2)}-H_k^{(12)}\vartheta_k^{(1)}-\tilde{\alpha}_k\Delta \vartheta_k^{(2)}\\
              \end{pmatrix},
\end{equation*}
which written in a matrix format just means
\begin{equation}\label{g2pal2}
  \vartheta_k+H_k(- \tilde{\alpha}_{k}\vartheta_k)
  \quad \mbox{is parallel to} \quad
  H_k^{-1}\vartheta_k - \tilde{\alpha}_{k} \vartheta_k.
\end{equation}
Since $n=2$, we have $B_k B_k^T =B_k^T B_k=I$.
Then, it follows from  $g_k = B_k^T \vartheta_k$, $\vartheta_k = B_k g_k$,
 $H_k = B_k A B_k^T$ and $g_{k+1} = g_k - \tilde{\alpha}_k A g_k$  that
$g_{k+1}= B_k^T \vartheta_k+ B_k^T H_k(- \tilde{\alpha}_{k}\vartheta_k)$.
So, we have from \eqref{g2pal2} that \eqref{parallel-g2} holds.
Therefore, \eqref{new2} implies \eqref{parallel-g2} holds.

Now, it follows from \eqref{eqgrad} and $H_{k}^{-1} = B_k A^{-1} B_k^T$  that
\[
- B_k^T  H_{k}^{-1} \vartheta_k + \tilde{\alpha}_{k} g_{k}
= -  A^{-1} g_{k} + \tilde{\alpha}_{k} g_{k} = - A^{-1}(g_{k}- \tilde{\alpha}_{k} A g_{k})
= - A^{-1} g_{k+1}.
\]
Hence, \eqref{parallel-g2} implies $g_{k+1}$ is parallel to $A^{-1} g_{k+1}$.
So, if $\tilde{\alpha}_{k}$ given by \eqref{new2} is used at the $k$-th iteration,
then $g_{k+1}$ is parallel to $A^{-1} g_{k+1}$. Since $x_{k+1}$ is not the minimizer, we have
$g_{k+1} \ne 0$. So, $g_{k+1}$ is an eigenvector of $A$, i.e. $Ag_{k+1} = \lambda g_{k+1}$
for some $\lambda > 0$. Since $x_{k+2}$ is not the minimizer, we have $g_{k+2} \ne 0 $ and
the algorithm will not stop at the $k+2$-th iteration. So, by  \eqref{glstep1}, we have
$\alpha_{k+2}=\frac{g_{k+1}^T \Psi(A) g_{k+1}}{g_{k+1}^T \Psi(A)A g_{k+1}} = 1/\lambda$.
Hence, we have $g_{k+3}=(1-\alpha_{k+2}\lambda)g_{k+2} =0$, which implies $x_{k+3}$
 must be the minimizer. We complete the proof.
\qed
\end{proof}

Notice that by setting $k_0 =2$ in the above Theorem \ref{thftm},
the new gradient method in Theorem \ref{thftm} will find the exact minimizer in at most $5$ iterations
when minimizing a two-dimensional strongly convex quadratic function.
In fact, since $\Delta=\lambda_1\lambda_2$ and $H_k^{(11)}+H_k^{(22)}=\lambda_1+\lambda_2$,
the equation \eqref{eqalp2} has two positive roots $1/\lambda_1$ and $1/\lambda_2$.
This observation allows us to use the stepsize $\tilde{\alpha}_{k_0}$ with some retards
as stated in the following corollary, which would lead us a more convenient way for
choosing stepsizes when the objective function is not quadratic.

\begin{corollary}
Suppose that a gradient method is applied to a two-dimensional quadratic function \eqref{qudpro}
 with $\alpha_{k_0+m}=\tilde{\alpha}_{k_0}$ for $k_0\geq2$ and some positive integer $m$,
 and $\alpha_k$ given by \eqref{glstep1} for all $k \neq k_0+m$.
Then, the method stops in at most $k_0+m+3$ iterations.
\end{corollary}

By setting $\Psi(A)=I$, $\Psi(A)=A$ and $r=1/2$ in \eqref{eqhes}, and setting $k_0=k$
in \eqref{new2}, we can derive the following two stepsizes:
\begin{equation}\label{newsbb1}
  \tilde{\alpha}_k^{BB1}=  \frac{2}{\frac{q_{k-1}^TAq_{k-1}}{\|q_{k-1}\|^2} + \frac{1}{\alpha_k^{SD}}+
  \sqrt{\left(\frac{q_{k-1}^TAq_{k-1}}{\|q_{k-1}\|^2} - \frac{1}{\alpha_k^{SD}}\right)^2+\frac{4(q_{k-1}^TAg_k)^2}{\|q_{k-1}\|^2\|g_k\|^2}}}
\end{equation}
and
\begin{equation}\label{newsbb2}
  \tilde{\alpha}_{k}^{BB2}=
  \frac{2}{\frac{1}{\hat{\alpha}_{k-1}}
  +\frac{1}{\alpha_k^{MG}}+
  \sqrt{\left(\frac{1}{\hat{\alpha}_{k-1}}-\frac{1}{\alpha_k^{MG}}\right)^2
  +\Gamma_k}},
\end{equation}
respectively, where
\begin{equation}\label{snew21}
  \hat{\alpha}_{k}=\frac{q_{k}^TAq_{k}}{q_{k}^TA^2q_{k}}~~\textrm{and}~~
  \Gamma_k=\frac{4(q_{k-1}^TA^2g_k)^2}{q_{k-1}^TAq_{k-1}\cdot g_k^TA g_k}.
\end{equation}
By \eqref{newsbb1} and \eqref{newsbb2}, we have
\begin{equation}\label{upbdtalp3}
  \tilde{\alpha}_k^{BB1}\leq\min\left\{\alpha_k^{SD},\frac{\|q_{k-1}\|^2}{q_{k-1}^TAq_{k-1}}\right\}
  ~~\textrm{and}~~
  \tilde{\alpha}_{k}^{BB2}\leq\min\{\alpha_k^{MG},\hat{\alpha}_{k-1}\}.
\end{equation}
Hence, both $\tilde{\alpha}_k^{BB1}$ and $\tilde{\alpha}_k^{BB2}$ are short monotone steps
for reducing the value and gradient norm of the objective function, respectively.
And it follows from Theorem \ref{thftm} that by inserting the monotone steps
$\tilde{\alpha}_k^{BB1}$ and $\tilde{\alpha}_k^{BB2}$ into
the BB1 and BB2 methods, respectively, the gradient method will have finite termination
for minimizing two-dimensional strongly convex quadratic functions.

To numerically verify this finite termination property, we apply the method \eqref{glstep1} with $\Psi(A)=I$ (i.e., the BB1 method) and $\tilde{\alpha}_2^{BB1}$ given by \eqref{newsbb1} to minimize a two-dimensional quadratic function \eqref{qudpro} with
\begin{equation}\label{twoquad}
  A=\textrm{diag}\{1,\lambda\} \quad \mbox{and} \quad b=0.
\end{equation}
We run the algorithm for five iterations using ten random starting points.
The averaged values of $\|g_5\|$ and $f(x_5)$ are presented in Table \ref{tb2ft}. Moreover, we also run
 BB1 method for a comparison purpose. We can observe that for different values of $\lambda$, the values of $\|g_5\|$ and $f(x_5)$ obtained by BB1 method with $\tilde{\alpha}_2^{BB1}$ given by \eqref{newsbb1} are numerically very close to zero. However, even for the case $\lambda=10$, $\|g_5\|$ and $f(x_5)$ obtained by pure BB1 method are far away from zero. These numerical results coincide with our analysis and show that the nonmonotone method \eqref{glstep1} can be significantly accelerated by incorporating proper monotone steps.

\linespread{1.2}
\begin{table}[ht!b]
\setlength{\tabcolsep}{2ex}
{\footnotesize
\caption{Averaged results on problem \eqref{twoquad} with different condition numbers}\label{tb2ft}
\begin{center}
\begin{tabular}{|c|c|c|c|c|}
\hline
 \multicolumn{1}{|c|}{\multirow{2}{*}{$\lambda$}}
&\multicolumn{2}{c|}{\multirow{1}{*}{BB1}}
&\multicolumn{2}{c|}{\multirow{1}{*}{BB1 with $\tilde{\alpha}_2^{BB1}$ given by \eqref{newsbb1}}}\\
\cline{2-5}
 &\multicolumn{1}{c|}{$\|g_5\|$} &\multicolumn{1}{c|}{$f(x_5)$}
&\multicolumn{1}{c|}{$\|g_5\|$} &\multicolumn{1}{c|}{$f(x_5)$}\\
 \hline
10  &   6.6873e+00  &   4.8701e+00  &   1.1457e-16  &   5.8735e-31\\
 \hline
100  &   8.6772e+01  &   1.9969e+02  &   2.1916e-16  &   2.8047e-30\\
\hline
1000  &   2.0925e+02  &   9.7075e+01  &   3.4053e-19  &   2.3730e-29\\
\hline
10000  &   1.7943e+02  &   9.6935e+00  &   5.1688e-19  &   6.7870e-28\\
\hline
\end{tabular}
\end{center}
}
\end{table}

\section{New methods}\label{secarbbmd}
In this section, based on the above analysis, we propose an adaptive nonmonotone gradient method (ANGM)
and its two variants, ANGR1 and ANGR2, for solving both unconstrained and box constrained optimization.
These new gradient methods adaptively take some nonmonotone steps involving the long and short
 BB stepsizes \eqref{sBB}, and some monotone steps using the new stepsize developed in the previous section.

\subsection{Quadratic case}

As mentioned in Section \ref{submot}, the stepsizes $\alpha_{k}^{BB1}$ generated by the BB1 method may be far away from the reciprocals of
the largest eigenvalues of the Hessian matrix $A$ of the quadratic function \eqref{qudpro}.
In other words, the stepsize $\alpha_{k}^{BB1}$ may be too large to effectively decrease the
components of gradient $g_k$ corresponding to the first several largest eigenvalues,
which, by \eqref{updgki}, can be greatly reduced when small stepsizes are employed.
In addition, it has been observed by many works in the recent literature that gradient methods using long and short stepsizes adaptively generally
perform much better than those using only one type of stepsizes, for example see \cite{dai2003altermin,de2014efficient,di2018steplength,gonzaga2016steepest,huang2019gradient,huang2019asymptotic}.
So, we would like to develop gradient methods that combines the two nonmonotone BB stepsizes
with the short monotone stepsize given by \eqref{new2}.

We first extend the orthogonal property developed in Lemma \ref{orpty} and the finite termination result given in Theorem \ref{thftm}.

\begin{lemma}[Generalized orthogonal property]\label{gorpty}
Suppose that a gradient method \eqref{eqitr} with stepsizes in the form of \eqref{glstep1} is applied to minimize a quadratic function \eqref{qudpro}.
In particular, at the $k-1$-th and $k$-th iteration, two stepsizes $\alpha_{k-1}(\Psi(A))$ and $\alpha_k(\Psi_1(A))$ are used,
respectively, where  $\Psi$ and $\Psi_1$ may be two different analytic functions used in \eqref{glstep1}.
If $q_k \in \mathbb{R}^n$ satisfies
\begin{equation}\label{gdefqk}
  (I-\alpha_{k-1}(\Psi(A))A)q_k=g_{k-1},
\end{equation}
then we have
  \begin{equation}\label{gorthg}
  q_k^T\Psi_1(A)g_{k+1}=0.
\end{equation}
\end{lemma}
\begin{proof}
Notice that by  \eqref{eqitr}, we have
\begin{equation*}
  g_k=g_{k-1}-\alpha_{k-1}(\Psi(A))Ag_{k-1} \quad \mbox{and} \quad g_{k+1}=g_k-\alpha_k(\Psi_1(A))Ag_k.
\end{equation*}
Then, the proof is essential the same as those in the proof of Lemma \ref{orpty}. \qed
\end{proof}

Based on Lemma \ref{gorpty} and using the same arguments as those in the proof of Theorem \ref{thftm}, we can obtain the following finite termination result
even different function $\Psi$'s are used in \eqref{glstep1} to obtain the stepsizes.
\begin{theorem}[Generalized finite termination]\label{gthftm}
Suppose that a gradient method \eqref{eqitr} is applied to minimize a two-dimensional quadratic function \eqref{qudpro} with $\alpha_k$ given by \eqref{glstep1}
for all $k \neq k_0$ and $k \ne k_0-1$, and uses the stepsizes $\alpha_{k-1}(\Psi_1(A))$ and $\alpha_k(\Psi_1(A))$ at the $k-1$-th and $k$-th iteration, respectively,
where $k_0 \geq2$.  Then, the method will find the minimizer in at most $k_0+3$ iterations.
\end{theorem}

Theorem \ref{gthftm} allows us to incorporate the nonmonotone BB stepsizes $\alpha_{k}^{BB1}$ and  $\alpha_{k}^{BB2}$, and the short monotone stepsize
$\tilde{\alpha}_{k}^{BB2}$ in one gradient method.
Alternate or adaptive scheme has been employed for choosing long and short stepsizes in BB-type methods \cite{dai2005projected,zhou2006gradient}.
And recent studies show that adaptive strategies are more preferred  than the alternate scheme \cite{dhl2018,zhou2006gradient}.
Hence, we would like develop adaptive strategies to choose proper stepsizes for our new gradient methods.
In particular, our adaptive nonmonotone gradient method (ANGM) takes the long BB stepsize $\alpha_{k}^{BB1}$ when $\alpha_{k}^{BB2}/\alpha_{k}^{BB1}\geq\tau_1$ for some $\tau_1\in(0,1)$.
Otherwise, a short stepsize  $\alpha_{k}^{BB2}$ or $\tilde{\alpha}_{k}^{BB2}$ will be taken depending on the ratio $\|g_{k-1}\|/\|g_{k}\|$.
Notice that $\alpha_{k}^{BB2}$ minimizes the gradient in the sense that
\begin{equation*}
  \alpha_{k}^{BB2}=\alpha_{k-1}^{MG}=\arg\min_{\alpha \in \mathbb{R}} \|g_{k-1}-\alpha Ag_{k-1}\|.
\end{equation*}
So, when $\|g_{k-1}\|/\|g_{k}\| > \tau_2 $ for some $\tau_2 >1$, i.e. the gradient norm decreases, the previous stepsize $\alpha_{k-1}$
is often a reasonable approximation of  $\alpha_{k}^{BB2}$.
By our numerical experiments, when BB method is applied the searches are often dominated in some
 two-dimensional subspaces.
And the new gradient method in Theorem \ref{gthftm} would have finite convergence for minimizing two-dimensional quadratic function
when the new stepsize $\tilde{\alpha}_{k}^{BB2}$ is applied after some BB2 steps.
Hence, our ANGM would employ the new monotone stepsize $\tilde{\alpha}_{k}^{BB2}$ when $\|g_{k-1}\|\geq\tau_2\|g_{k}\|$;
otherwise, certain BB2 steps should be taken. In practice, we find that when $\|g_{k-1}\| \le \tau_2\|g_{k}\|$, ANGM often
has good performance by taking the stepsize  $\min\{\alpha_{k}^{BB2},\alpha_{k-1}^{BB2}\}$.
To summarize, our ANGM applies the following adaptive strategies for choosing stepsizes:
\begin{equation}\label{newmd0}
\alpha_{k}=
\begin{cases}
\min\{\alpha_{k}^{BB2},\alpha_{k-1}^{BB2}\},& \text{if $\alpha_{k}^{BB2}<\tau_1\alpha_{k}^{BB1}$ and $\|g_{k-1}\|<\tau_2\|g_{k}\|$}; \\
\tilde{\alpha}_{k}^{BB2},& \text{if $\alpha_{k}^{BB2}<\tau_1\alpha_{k}^{BB1}$ and $\|g_{k-1}\|\geq\tau_2\|g_{k}\|$}; \\
\alpha_{k}^{BB1},& \text{otherwise},
\end{cases}
\end{equation}

Notice that the calculation of $\tilde{\alpha}_{k}^{BB2}$ needs to compute $\alpha_k^{MG}$ which is not easy to obtain when the objective function is not quadratic.
In stead, the calculation of $\tilde{\alpha}_{k-1}^{BB2}$ will just require $\alpha_k^{BB2}$,
which is readily available even for general objective function. Moreover, it is found in recent research that
gradient methods using retard stepsizes can often lead better performances \cite{friedlander1998gradient}.
Hence, in the first variant of ANGM, we simply replace $\tilde{\alpha}_{k}^{BB2}$ in \eqref{newmd0} by
$\tilde{\alpha}_{k-1}^{BB2}$, i.e. the stepsizes are chosen as
\begin{equation}\label{newmd1}
\alpha_{k}=
\begin{cases}
\min\{\alpha_{k}^{BB2},\alpha_{k-1}^{BB2}\},& \text{if $\alpha_{k}^{BB2}<\tau_1\alpha_{k}^{BB1}$ and $\|g_{k-1}\|<\tau_2\|g_{k}\|$}; \\
\tilde{\alpha}_{k-1}^{BB2},& \text{if $\alpha_{k}^{BB2}<\tau_1\alpha_{k}^{BB1}$ and $\|g_{k-1}\|\geq\tau_2\|g_{k}\|$}; \\
\alpha_{k}^{BB1},& \text{otherwise}.
\end{cases}
\end{equation}
We call the gradient method using stepsize \eqref{newmd1} ANGR1.
On the other hand, since the calculation of $\tilde{\alpha}_{k-1}^{BB2}$ also needs $\hat{\alpha}_{k-2}$ and $\Gamma_{k-1}$ and
by \eqref{upbdtalp3},
\begin{equation}\label{upbdtalp2}
  \tilde{\alpha}_{k-1}^{BB2}\leq\min\{\alpha_k^{BB2},\hat{\alpha}_{k-2}\},
\end{equation}
to simplify ANGR1, we may further replace  $\tilde{\alpha}_{k-1}^{BB2}$ in \eqref{newmd1} by its upper bound in \eqref{upbdtalp2}.
As a result, we have the second variant of ANGM, which chooses stepsizes as
\begin{equation}\label{newmd2}
\alpha_{k}=
\begin{cases}
\min\{\alpha_{k}^{BB2},\alpha_{k-1}^{BB2}\},& \text{if $\alpha_{k}^{BB2}<\tau_1\alpha_{k}^{BB1}$ and $\|g_{k-1}\|<\tau_2\|g_{k}\|$}; \\
\min\{\alpha_k^{BB2},\hat{\alpha}_{k-2}\},& \text{if $\alpha_{k}^{BB2}<\tau_1\alpha_{k}^{BB1}$ and $\|g_{k-1}\|\geq\tau_2\|g_{k}\|$}; \\
\alpha_{k}^{BB1},& \text{otherwise}.
\end{cases}
\end{equation}
We call the gradient method using stepsize \eqref{newmd2} ANGR2.

In terms of global convergence for minimizing quadratic function \eqref{qudpro},
by \eqref{upbdtalp3}, we can easily show the $R$-linear global convergence of ANGM since it satisfies the property in \cite{dai2003alternate}.
Similarly, $R$-linear convergence of ANGR1 and ANGR2 can be also established. See the proof of Theorem~3 in \cite{dhl2018} for example.

\begin{remark}
Compared with other gradient methods, ANGM, ANGR1 and ANGR2 do not need additional matrix-vector products.
In fact, it follows from \eqref{defqk} that $Aq_k=\frac{1}{\alpha_{k-1}}(q_k-g_{k-1})$,
which gives
\begin{equation}\label{snew2}
  \hat{\alpha}_{k}=\frac{q_{k}^TAq_{k}}{q_{k}^TA^2q_{k}}
=\frac{\alpha_{k-1}q_{k}^T(q_{k}-g_{k-1})}
{(q_{k}-g_{k-1})^T(q_{k}-g_{k-1})}.
\end{equation}
Hence,  no additional matrix-vector products are needed for calculation of
$\hat{\alpha}_{k-1}$ in $\tilde{\alpha}_{k}^{BB2}$, $\hat{\alpha}_{k-2}$ in $\tilde{\alpha}_{k-1}^{BB2}$
and the stepsize used in ANGR2.
Since the calculation of $Ag_k$ is necessary for the calculation of $g_{k+1}$, $\Gamma_k$ in $\tilde{\alpha}_{k}^{BB2}$ requires no additional matrix-vector products either.
 As for $\tilde{\alpha}_{k-1}^{BB2}$, it follows from
$g_{k-1}^TA^2q_{k-2}=\frac{1}{\alpha_{k-3}}(q_{k-2}-g_{k-3})^T A g_{k-1}$ and $Ag_{k-1}=\frac{1}{\alpha_{k-1}}(g_{k-1}-g_{k})$ that
\begin{align}
  \Gamma_{k-1}&
=\frac{4((q_{k-2}-g_{k-3})^T Ag_{k-1})^2}{\alpha_{k-3}((q_{k-2}-g_{k-3})^Tq_{k-2})\cdot g_{k-1}^TA g_{k-1}}\nonumber\\
&=\frac{4((q_{k-2}-g_{k-3})^T(g_{k-1}-g_{k}))^2}
{\alpha_{k-3}\alpha_{k-1}((q_{k-2}-g_{k-3})^Tq_{k-2})\cdot g_{k-1}^T(g_{k-1}-g_{k})}.
\end{align}
Thus, no additional matrix-vector products are required for calculation of $\Gamma_{k-1}$
 in $\tilde{\alpha}_{k-1}^{BB2}$.
\end{remark}

\begin{remark}
Notice that all the new methods, ANGM, ANGR1 and ANGR2, require the vector $q_{k}$ for calculation
of their stepsizes.
However, computing $q_k$ exactly from \eqref{defqk} maybe as difficult as minimizing the quadratic function.
Notice that the $q_k$ satisfying \eqref{defqk} also satisfies the secant equation
\begin{equation}\label{secant-cond}
q_k^Tg_{k}=\|g_{k-1}\|^2.
\end{equation}
Hence, we may find an approximation of $q_k$ by requiring the above secant condition holds.
One efficient way to find such a $q_k$ satisfying the secant equation \eqref{secant-cond} is to simply
 treat the Hessian matrix $A$ as the diagonal matrix \eqref{formA} and
derive $q_k$ from \eqref{defqk}, that is when $g_k^{(i)} \ne 0$,
\begin{equation}\label{defq}
  q_k^{(i)}=\frac{g_{k-1}^{(i)}}{1-\alpha_{k-1}\lambda_i}=
  \frac{(g_{k-1}^{(i)})^2}{g_k^{(i)}},~~i=1,\ldots,n,
\end{equation}
And we can just let $q_k^{(i)} = 0$, if  $g_k^{(i)}=0$.
To summarize, the approximated $q_k$ can be computed by
  \begin{equation}\label{defq2}
  q_k^{(i)}=
\left\{
              \begin{array}{ll}
                \frac{(g_{k-1}^{(i)})^2}{g_k^{(i)}}, & \hbox{if $g_k^{(i)}\neq0$;} \\
                0, & \hbox{if $g_k^{(i)}=0$.}
              \end{array}
            \right.
\end{equation}
As we will see in Section \ref{secnum}, this simple way of calculating $q_k$
leads very efficient algorithm.
\end{remark}

For a simple illustration of numerical behavior of ANGR1, we again applied ANGR1 with
$\tau_1=0.85$ and $\tau_2=1.3$ to solve problem \eqref{pro2} with $n=10$.
Fig. \ref{bbabbrg1} shows the largest component $|g_k^{(i_1)}|$ of the gradient
generated by  BB1 and ANGR1 methods
against the iteration number, where circle means the ANGR1 method takes the new stepsize
 $\tilde{\alpha}_{k}^{BB2}$ at that iteration. It can be seen that $|g_k^{(i_1)}|$ generated by
  BB1 method often increases significantly with a much larger value at the iteration where the new
stepsize $\tilde{\alpha}_{k}^{BB2}$ is applied. On the other hand,  $|g_k^{(i_1)}|$ generated by the ANGR1 method
is often reduced and kept small after the new stepsize $\tilde{\alpha}_{k}^{BB2}$ is applied.
A detail correspondence of $i_1$ and $\lambda_j$ is presented in Table \ref{bb1angmnj},
where $n_j$ is the total number of  $i_1's$ for which $i_1=j$, $j=1,\ldots,10$.
We can see from the last three columns in Table \ref{bb1angmnj} that
ANGR1 is much efficient than BB1 for decreasing those components of $g_k$ corresponding to large eigenvalues.
Hence, the undesired behavior of BB1 discussed in the motivation Section~\ref{submot}
is greatly eliminated by ANGR1.

\begin{figure}[thp!b]
  \centering
  \includegraphics[width=0.65\textwidth,height=0.45\textwidth]{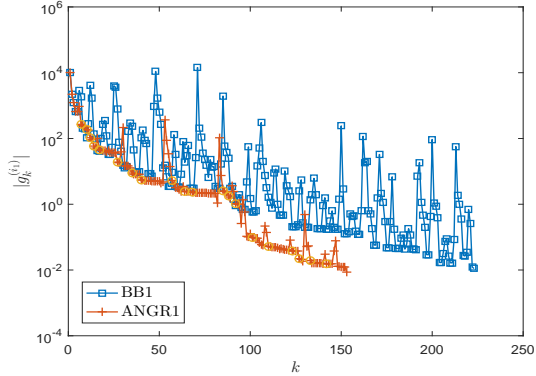}
  \caption{Problem \eqref{pro2} with $n=10$: history of $|g_k^{(i_1)}|$ generated by the BB1 and ANGR1 methods}\label{bbabbrg1}
\end{figure}

\begin{table}[ht!b]
{
\caption{The correspondence of $i_1$ and $\lambda_j$ by the BB1 and ANGR1 methods}\label{bb1angmnj}
\begin{center}
\begin{tabular}{|c|c|c|c|c|c|c|c|c|c|c|}
\hline
$n_j$ &1 &2 &3 &4 &5 &6 &7 &8 &9 &10 \\
 \hline
BB1 &40 &5 &6 &4 &6 &2 &10 &21 &42 &87 \\
 \hline
ANGR1 &51 &4 &22 &18 &16 &4 &1 &8 &12 &17 \\
\hline
\end{tabular}
\end{center}
}
\end{table}

\subsection{Bound constrained case}
In this subsection, we would like to extend ANGR1 and ANGR2 methods for solving the bound constrained optimization
\begin{equation}\label{conprob}
\min_{x\in\Omega} ~~f(x),
\end{equation}
where $f$ is Lipschitz continuously differentiable on the feasible set
$\Omega=\{x\in\mathbb{R}^n|~l\leq x\leq u\}$. Here, $l \le x \le u$ means componentwise
$l_i \le x_i \le u_i$ for all $i=1, \ldots, n$.
Clearly, when $l_i=-\infty$ and $u_i=+\infty$ for all $i$, problem \eqref{conprob} reduces to the smooth unconstrained optimization.

Our methods will incorporate the gradient projection strategy and update the iterates as
\[
  x_{k+1}=x_{k}+\lambda_{k}d_{k},
\]
with $\lambda_{k}$ being a step length determined by some line searches and $d_k$ being the search direction given by
\begin{equation}\label{dirc}
  d_k=P_{\Omega}(x_k-\alpha_kg_k)-x_k,
\end{equation}
where $P_{\Omega}(\cdot)$ is the Euclidean projection onto $\Omega$ and $\alpha_k$ is our proposed stepsize.

It is well-known that the components of iterates generated by gradient descent methods
 corresponding to optimal solutions at the boundary
will be finally unchanged when the problem is nondegenerate. Hence, in \cite{huang2019gradient},
the authors suggest to use the following modified BB stepsizes for bound constrained optimization
\begin{equation}\label{bsbb-1-2}
 \bar{\alpha}_k^{BB1}  =\frac{s_{k-1}^Ts_{k-1}}{s_{k-1}^T\bar{y}_{k-1}} \quad \mbox{and} \quad  \bar{\alpha}_k^{BB2}=\frac{s_{k-1}^T\bar{y}_{k-1}}{\bar{y}_{k-1}^T\bar{y}_{k-1}},
\end{equation}
where $\bar{y}_{k-1}$ is given by
\begin{equation} \label{bary}
  \bar{y}_{k-1}^{(i)}=\left\{
                    \begin{array}{ll}
                      0, & \hbox{ if $s_{k-1}^{(i)}=0$;} \\
                      g_k^{(i)}-g_{k-1}^{(i)}, & \hbox{ otherwise.}
                    \end{array}
                  \right.
\end{equation}
Notice that $\alpha_k^{BB1} = \bar{\alpha}_k^{BB1}$. We will also do this modifications
for solving bound constrained optimization and
replace the two BB stepsizes in our new methods by $\bar{\alpha}_k^{BB1}$ and $\bar{\alpha}_k^{BB2}$.

As mentioned before, we expect to get short steps using our new stepsizes. Since \eqref{upbdtalp2} may not hold for general functions, we would impose $\bar{\alpha}_{k}^{BB2}$ as a safeguard.
As a result, our ANGR1 and ANGR2 methods for solving bound constrained optimization employ the following stepsizes:
\begin{equation}\label{newgmd1}
\bar{\alpha}_{k}=
\begin{cases}
\min\{\bar{\alpha}_{k-1}^{BB2},\bar{\alpha}_{k}^{BB2}\},& \text{if $\bar{\alpha}_{k}^{BB2}<\tau_1\bar{\alpha}_{k}^{BB1}$ and $\|\bar{g}_{k-1}\|<\tau_2\|\bar{g}_{k}\|$}; \\
\min\{\bar{\alpha}_{k}^{BB2},\tilde{\alpha}_{k-1}^{BB2}\},& \text{if $\bar{\alpha}_{k}^{BB2}<\tau_1\bar{\alpha}_{k}^{BB1}$ and $\|\bar{g}_{k-1}\|\geq\tau_2\|\bar{g}_{k}\|$}; \\
\bar{\alpha}_{k}^{BB1},& \text{otherwise},
\end{cases}
\end{equation}
and
\begin{equation}\label{newgmd2}
\bar{\alpha}_{k}=
\begin{cases}
\min\{\bar{\alpha}_{k-1}^{BB2},\bar{\alpha}_{k}^{BB2}\},& \text{if $\bar{\alpha}_{k}^{BB2}<\tau_1\bar{\alpha}_{k}^{BB1}$ and $\|\bar{g}_{k-1}\|<\tau_2\|\bar{g}_{k}\|$}; \\
\min\{\bar{\alpha}_{k}^{BB2},\hat{\alpha}_{k-2}\},& \text{if $\bar{\alpha}_{k}^{BB2}<\tau_1\bar{\alpha}_{k}^{BB1}$ and $\|\bar{g}_{k-1}\|\geq\tau_2\|\bar{g}_{k}\|$}; \\
\bar{\alpha}_{k}^{BB1},& \text{otherwise},
\end{cases}
\end{equation}
respectively, where $\tau_1\in(0,1)$, $\tau_2\geq1$, and $\bar{g}_{k}=P_{\Omega}(x_{k}-g_{k})-x_{k}$.

The overall algorithm of ANGR1 and ANGR2 for solving bound constrained optimization \eqref{conprob} are given in Algorithm \ref{al1},
where the adaptive nonmonotone line search by Dai and Zhang \cite{dai2001adaptive} is employed to ensure global convergence
and achieve better numerical performance. In particular, the step length $\lambda_{k} = 1$ is accepted if
\begin{equation} \label{nonmls}
f(x_{k}+d_{k})\leq f_{r} +\sigma g_{k}^Td_{k},
\end{equation}
where $f_{r}$ is the so-called reference function value and is adaptively updated by the rules given in \cite{dai2001adaptive} and  $\sigma\in(0,1)$ is a line search parameter. Once \eqref{nonmls} is not accepted, it will perform
an Armijo-type back tracking line search to find the step length $\lambda_{k}$ such that
\begin{equation} \label{nonmls-2}
f(x_{k}+\lambda_{k}d_{k})\leq \min\{f_{\max},f_{r}\}+\sigma\lambda_{k}g_{k}^Td_{k},
\end{equation}
where $f_{\max}$ is the maximal function value in recent $M$ iterations, i.e.,
\begin{equation*}
  f_{\max}=\max_{0\leq i\leq \min\{k,M-1\}} f(x_{k-i}).
\end{equation*}
This nonmonotone line search is observed specially suitable for BB-type methods \cite{dai2001adaptive}. Moreover, under standard assumptions, Algorithm \ref{al1} ensures convergence in the sense that $\lim\inf_{k\rightarrow\infty}\|\bar{g}_{k}\|=0$, see \cite{hager2006new}.

\begin{algorithm}[h]
\caption{Adaptive nonmonotone gradient method with retard steps}\label{al1}

\KwIn{$x_{0}\in \mathbb{R}^n$, $\epsilon,\sigma\in(0,1)$, $\tau>0$, $M\in\mathbb{N}$, $0<\alpha_{\min}\leq\alpha_{\max}$, $\alpha_{0}\in[\alpha_{\min},\alpha_{\max}]$, $k:=0$.}

\While{$\|\bar{g}_k\|>\epsilon$}{

 Compute the search direction $d_k$ by \eqref{dirc};

 Determine $\lambda_{k}$ by Dai-Zhang nonmonotone line search;

 $x_{k+1}=x_{k}+\lambda_{k}d_{k}$;

Compute $s_{k}=x_{k+1}-x_{k}$ and $\bar{y}_{k}$ by \eqref{bary};

\uIf{$s_{k}^T\bar{y}_{k}>0$}{
%
%
Compute stepsize $\bar{\alpha}_{k+1}$ by \eqref{newgmd1} or \eqref{newgmd2};

 $\alpha_{k+1} = \max\{\alpha_{\min},\min\{\bar{\alpha}_{k+1},\alpha_{\max}\}\};$
}
\Else
{
 $\alpha_{k+1} = 1/\|\bar{g}_{k+1}\|_\infty;$
}
$k:=k+1;$
}
\end{algorithm}

\section{Numerical results}\label{secnum}
In this section, we present numerical comparisons of ANGM, ANGR1, ANGR2
with some recent very successful gradient descent methods on solving
quadratic, general unconstrained and bound constrained problems.
All the comparison methods were implemented in Matlab (v.9.0-R2016a) and run on a laptop with an Intel Core i7, 2.9 GHz processor and 8 GB of RAM running Windows 10 system.

\subsection{Quadratic problems}

In this subsection, we compare ANGM, ANGR1 and ANGR2 with the BB1 \cite{barzilai1988two}, DY \cite{dai2005analysis}, ABBmin2 \cite{frassoldati2008new}, and SDC \cite{de2014efficient}
methods on solving quadratic optimization problems.

We first solve some randomly generated quadratic problems from \cite{yuan2006new}.
Particularly, we solve
 \begin{equation}\label{quad-test1}
 \min_{x \in \mathbb{R}^n} \: f(x)=(x-x^*)^TV(x-x^*),
\end{equation}
 where $x^*$ is randomly generated with components between $-10$ and $10$, $V=\textrm{diag}\{v_1,\ldots,v_n\}$ is a diagonal matrix with
 $v_1=1$ and $v_n=\kappa$, and $v_j$, $j=2,\ldots,n-1$, are generated by the \emph{rand} function between 1 and $\kappa$.

\begin{table}[th!b]
\caption{Distributions of $v_j$}\label{tbspe}
\centering
\begin{tabular}{|c|c|}
\hline
 \multirow{1}{*}{Problem} &\multicolumn{1}{c|}{Spectrum} \\
\hline
\multirow{1}{*}{1} &$\{v_2,\ldots,v_{n-1}\}\subset(1,\kappa)$	\\
\hline
 \multirow{2}{*}{2}
&$\{v_2,\ldots,v_{n/5}\}\subset(1,100)$	\\
&$\{v_{n/5+1},\ldots,v_{n-1}\}\subset(\frac{\kappa}{2},\kappa)$	\\
\hline
\multirow{2}{*}{3}
&$\{v_2,\ldots,v_{n/2}\}\subset(1,100)$	\\
&$\{v_{n/2+1},\ldots,v_{n-1}\}\subset(\frac{\kappa}{2},\kappa)$	\\
\hline
\multirow{2}{*}{4}
&$\{v_2,\ldots,v_{4n/5}\}\subset(1,100)$	\\
&$\{v_{4n/5+1},\ldots,v_{n-1}\}\subset(\frac{\kappa}{2},\kappa)$	\\
\hline
\multirow{3}{*}{5}
&$\{v_2,\ldots,v_{n/5}\}\subset(1,100)$	\\
&$\{v_{n/5+1},\ldots,v_{4n/5}\}\subset(100,\frac{\kappa}{2})$	\\
&$\{v_{4n/5+1},\ldots,v_{n-1}\}\subset(\frac{\kappa}{2},\kappa)$	\\
\hline
\end{tabular}
\end{table}

We have tested five sets of problems \eqref{quad-test1} with $n=1,000$ using
 different spectral distributions of the Hessian listed in Table \ref{tbspe}.
The algorithm is stopped
once the number of iteration exceeds 20,000 or the gradient norm is reduced by a factor of $\epsilon$, which is set to $10^{-6}, 10^{-9}$ and $10^{-12}$, respectively.
Three different condition numbers $\kappa=10^4, 10^5$ and $10^6$ are tested. For each value of $\kappa$ or $\epsilon$, 10 instances of the problem are randomly generated
and the averaged results obtained by the starting point $x_0=(0,\ldots,0)^T$ are presented.
For the ABBmin2 method, $\tau$ is set to 0.9 as suggested in \cite{frassoldati2008new}. The parameter pair $(h,s)$ of the SDC method
is set to $(8,6)$ which is more efficient than other choices for this test. The $q_k$ is calculated by \eqref{defq2} for our methods.

\begin{figure}[ht!b]
\centering
    \includegraphics[width=0.6\textwidth,height=0.45\textwidth]{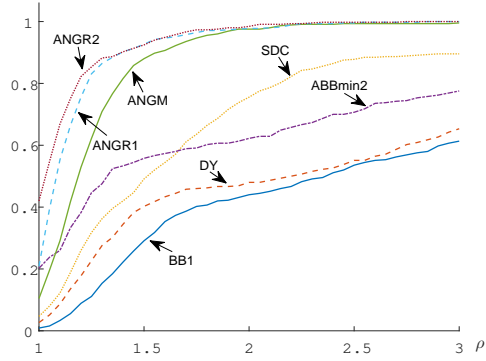}
  \caption{Performance profiles of compared methods on solving random quadratic problems  \eqref{quad-test1}
with spectral distributions in Table \ref{tbspe}, iteration metric}\label{frandp}
\end{figure}

\begin{table}[th!b]
\setlength{\tabcolsep}{0.15ex}
\caption{The numbers of averaged iterations of the ANGM, BB1, DY, ABBmin2 and SDC methods on solving quadratic problems \eqref{quad-test1}
with spectral distributions in Table \ref{tbspe}}\label{tbANGM}
\centering
\begin{tiny}
\begin{tabular}{|c|c|c|c|c|c|c|c|c|c|c|c|c|c|c|}
\hline
 \multicolumn{1}{|c|}{\multirow{2}{*}{set}} &\multicolumn{1}{c|}{\multirow{2}{*}{$\epsilon$}}
 &\multicolumn{9}{c|}{$\tau_1$ for ANGM} &\multirow{2}{*}{BB1} &\multirow{2}{*}{DY} &\multirow{2}{*}{ABBmin2} &\multirow{2}{*}{SDC}\\
\cline{3-11}
  \multicolumn{1}{|c|}{}& \multicolumn{1}{c|}{}  &0.1  &0.2  &0.3  &0.4  &0.5   &0.6   &0.7   &0.8  &0.9 & & & &\\
 \hline
\multicolumn{1}{|c|}{\multirow{3}{*}{1}	}																																				                                                  	
& 1e-06  &   200.6  &   198.6  &   213.3  &   193.6  &   204.1  &   196.5  &   211.1  &   219.0  &   230.3        &   221.0  &   197.1  &   \textbf{177.1}  &   183.2  \\
& 1e-09  &   665.0  &   678.1  &   677.5  &   937.3  &   718.7  &   768.5  &  1259.4  &  1198.1  &  1195.3        &  2590.0  &  2672.7  &   \textbf{428.8}  &  2029.6  \\
& 1e-12  &   939.7  &   941.6  &   901.5  &  1159.2  &   951.2  &  1043.7  &  1464.3  &  1339.0  &  1411.9        &  6032.9  &  6353.9  &   \textbf{560.3}  &  4087.2  \\
\hline

\multicolumn{1}{|c|}{\multirow{3}{*}{2}}
& 1e-06  &   140.3  &   \textbf{140.2}  &   143.1  &   161.2  &   172.8  &   171.1  &   217.3  &   265.8  &   310.3       &   311.2  &   261.5  &   302.5  &   160.6  \\
& 1e-09  &   \textbf{546.1}  &   590.5  &   641.1  &   779.5  &   850.5  &   891.9  &  1055.7  &  1161.3  &  1299.0       &  1665.4  &  1340.0  &  1321.1  &   735.1  \\
& 1e-12  &   \textbf{895.1}  &  1025.2  &  1098.1  &  1328.6  &  1446.6  &  1598.2  &  1739.6  &  1908.7  &  2107.0       &  2820.8  &  2434.6  &  2267.5  &  1346.8  \\
\hline

\multicolumn{1}{|c|}{\multirow{3}{*}{3}}
& 1e-06  &   \textbf{163.2}  &   170.1  &   177.3  &   193.3  &   216.7  &   231.1  &   283.3  &   318.4  &   363.9        &   388.4  &   329.4  &   356.3  &   235.5  \\
& 1e-09  &   \textbf{566.5}  &   640.0  &   680.7  &   811.0  &   970.8  &   986.0  &  1188.3  &  1218.7  &  1364.4        &  1783.1  &  1511.8  &  1470.8  &   818.1  \\
& 1e-12  &   \textbf{928.4}  &  1030.0  &  1163.7  &  1412.0  &  1575.2  &  1733.2  &  1973.6  &  2075.2  &  2191.1        &  2977.7  &  2780.2  &  2288.4  &  1310.6  \\
\hline

\multicolumn{1}{|c|}{\multirow{3}{*}{4}}
& 1e-06  &   \textbf{212.3}  &   213.7  &   237.1  &   259.0  &   254.7  &   291.8  &   365.3  &   431.4  &   475.1        &   500.5  &   431.3  &   519.0  &   262.8  \\
& 1e-09  &   \textbf{616.1}  &   655.7  &   759.7  &   885.5  &   956.4  &  1107.9  &  1232.4  &  1405.3  &  1533.3        &  1859.4  &  1659.9  &  1489.5  &   805.5  \\
& 1e-12  &   \textbf{996.0}  &  1078.4  &  1250.9  &  1452.1  &  1629.5  &  1786.8  &  2041.6  &  2179.0  &  2427.6        &  3051.5  &  2785.4  &  2383.9  &  1469.3  \\
\hline

\multicolumn{1}{|c|}{\multirow{3}{*}{5}}
& 1e-06  &   \textbf{623.1}  &   654.8  &   663.4  &   671.0  &   683.4  &   697.2  &   761.3  &   813.7  &   931.0        &   832.5  &   650.8  &   816.0  &   668.3  \\
& 1e-09  &  \textbf{2603.0}  &  2654.7  &  2847.6  &  2901.4  &  2936.9  &  3161.7  &  3228.6  &  3306.4  &  3807.2        &  4497.2  &  3185.5  &  2929.7  &  3274.5  \\
& 1e-12  &  \textbf{4622.7}  &  4675.2  &  4905.9  &  4634.2  &  4818.7  &  4944.7  &  5224.0  &  5480.7  &  5972.1        &  7446.7  &  7024.1  &  4808.7  &  5816.6  \\
\hline

\multicolumn{1}{|c|}{\multirow{3}{*}{total}}
& 1e-06   &  \textbf{1339.5}  &  1377.5  &  1434.2  &  1478.1  &  1531.7  &  1587.7  &  1838.3  &  2048.4  &  2310.6       &  2253.7  &  1870.1  &  2170.9  &  1510.4   \\
& 1e-09   &  \textbf{4996.7}  &  5219.0  &  5606.6  &  6314.7  &  6433.3  &  6916.0  &  7964.4  &  8289.7  &  9199.2       & 12395.0  & 10370.0  &  7640.0  &  7662.7   \\
& 1e-12   &  \textbf{8382.0}  &  8750.3  &  9320.1  &  9986.1  & 10421.3  & 11106.5  & 12443.0  & 12982.6  & 14109.7       & 22329.5  & 21378.3  & 12308.9  & 14030.4   \\
\hline

\end{tabular}
\end{tiny}
\end{table}

\begin{table}[th!b]
\setlength{\tabcolsep}{0.25ex}
\caption{The numbers of averaged iterations of the ANGR1, BB1, DY, ABBmin2 and SDC methods on problems in Table \ref{tbspe}}\label{tbANGR1}
\centering
\begin{tiny}
\begin{tabular}{|c|c|c|c|c|c|c|c|c|c|c|c|c|c|c|}
\hline
 \multicolumn{1}{|c|}{\multirow{2}{*}{set}} &\multicolumn{1}{c|}{\multirow{2}{*}{$\epsilon$}}
 &\multicolumn{9}{c|}{$\tau_1$ for ANGR1} &\multirow{2}{*}{BB1} &\multirow{2}{*}{DY} &\multirow{2}{*}{ABBmin2} &\multirow{2}{*}{SDC}\\
\cline{3-11}
  \multicolumn{1}{|c|}{}& \multicolumn{1}{c|}{}  &0.1  &0.2  &0.3  &0.4  &0.5   &0.6   &0.7   &0.8  &0.9 & & & &\\
 \hline
\multicolumn{1}{|c|}{\multirow{3}{*}{1}}																																				
& 1e-06  &   198.7  &   185.8  &   188.1  &   182.5  &   189.2  &   188.2  &   184.5  &   200.1  &   214.1      &   221.0  &   197.1  &   \textbf{177.1}  &   183.2  \\
& 1e-09  &   655.5  &   641.6  &   625.4  &   667.9  &   621.7  &   680.9  &   762.7  &   811.5  &  1012.2      &  2590.0  &  2672.7  &   \textbf{428.8}  &  2029.6  \\
& 1e-12  &   897.3  &   854.7  &   844.5  &   838.4  &   775.2  &   813.1  &   917.4  &   984.7  &  1172.7      &  6032.9  &  6353.9  &   \textbf{560.3}  &  4087.2  \\
\hline

\multicolumn{1}{|c|}{\multirow{3}{*}{2}}
& 1e-06  &   137.9  &   119.0  &   119.2  &   118.7  &   \textbf{115.3}  &   124.5  &   152.4  &   157.5  &   178.9     &   311.2  &   261.5  &   302.5  &   160.6  \\
& 1e-09  &   \textbf{466.5}  &   470.9  &   487.2  &   512.0  &   540.2  &   589.9  &   672.6  &   691.0  &   722.3     &  1665.4  &  1340.0  &  1321.1  &   735.1  \\
& 1e-12  &   788.1  &   \textbf{783.3}  &   802.1  &   835.3  &   908.1  &  1004.5  &  1074.2  &  1110.3  &  1215.2     &  2820.8  &  2434.6  &  2267.5  &  1346.8  \\
\hline

\multicolumn{1}{|c|}{\multirow{3}{*}{3}}
& 1e-06  &   164.5  &   149.6  &   \textbf{146.3}  &   153.0  &   157.8  &   166.8  &   183.3  &   208.9  &   229.1      &   388.4  &   329.4  &   356.3  &   235.5  \\
& 1e-09  &   507.5  &   \textbf{473.0}  &   516.6  &   541.2  &   563.3  &   609.0  &   682.5  &   747.2  &   850.5      &  1783.1  &  1511.8  &  1470.8  &   818.1  \\
& 1e-12  &   818.3  &   \textbf{801.5}  &   865.4  &   894.4  &   965.4  &  1046.0  &  1108.8  &  1200.2  &  1419.8      &  2977.7  &  2780.2  &  2288.4  &  1310.6  \\
\hline

\multicolumn{1}{|c|}{\multirow{3}{*}{4}}
& 1e-06  &   189.7  &   \textbf{172.2}  &   168.8  &   179.6  &   195.2  &   212.7  &   229.4  &   262.9  &   282.1      &   500.5  &   431.3  &   519.0  &   262.8  \\
& 1e-09  &   539.0  &   \textbf{519.2}  &   558.7  &   570.2  &   617.1  &   689.6  &   786.3  &   853.1  &   901.8      &  1859.4  &  1659.9  &  1489.5  &   805.5  \\
& 1e-12  &   \textbf{849.0}  &   851.1  &   894.2  &   934.8  &  1011.8  &  1065.5  &  1197.3  &  1291.8  &  1408.0      &  3051.5  &  2785.4  &  2383.9  &  1469.3  \\
\hline

\multicolumn{1}{|c|}{\multirow{3}{*}{5} }
& 1e-06  &   625.7  &   587.6  &   600.1  &   \textbf{583.1}  &   614.2  &   632.2  &   675.8  &   726.9  &   769.0      &   832.5  &   650.8  &   816.0  &   668.3  \\
& 1e-09  &  2539.9  &  \textbf{2518.1}  &  2612.7  &  2530.4  &  2636.0  &  2581.2  &  2773.6  &  2875.2  &  3026.2      &  4497.2  &  3185.5  &  2929.7  &  3274.5  \\
& 1e-12  &  \textbf{4207.3}  &  4238.2  &  4292.4  &  4256.2  &  4240.7  &  4285.1  &  4397.2  &  4534.8  &  4729.0      &  7446.7  &  7024.1  &  4808.7  &  5816.6  \\
\hline

\multicolumn{1}{|c|}{\multirow{3}{*}{total}}
& 1e-06   &  1316.5  &  \textbf{1214.1}  &  1222.6  &  1217.0  &  1271.8  &  1324.5  &  1425.3  &  1556.4  &  1673.1    &  2253.7  &  1870.1  &  2170.9  &  1510.4   \\
& 1e-09   &  4708.5  &  \textbf{4622.7}  &  4800.7  &  4821.7  &  4978.4  &  5150.6  &  5677.7  &  5977.9  &  6512.9    & 12395.0  & 10370.0  &  7640.0  &  7662.7   \\
& 1e-12   &  7560.1  &  \textbf{7528.8}  &  7698.7  &  7759.1  &  7901.3  &  8214.1  &  8694.9  &  9121.9  &  9944.7    & 22329.5  & 21378.3  & 12308.9  & 14030.4   \\
\hline

\end{tabular}
\end{tiny}
\end{table}

We compared the algorithms by using the performance profiles of Dolan and Mor\'{e} \cite{dolan2002} on iteration metric. In these performance profiles, the vertical axis shows the percentage of the problems the method solves within the factor $\rho$ of the metric used by the most effective method in this comparison.
Fig. \ref{frandp} shows the performance profiles of ANGM, ANGR1 and ANGR2 obtained by setting $\tau_1=0.1$, $\tau_2=1$ and other four compared methods.
Clearly, ANGR2 outperforms all other methods.
In general, we can see ANGM, ANGR1 and ANGR2 are much better than the BB1, DY and SDC methods.

To further analyze the performance of our methods, we present results of them with different values of $\tau_1$ from 0.1 to 0.9 in Tables \ref{tbANGM}, \ref{tbANGR1} and \ref{tbANGR2}. From Table \ref{tbANGM}, we can see that, for the first problem set, ANGM is much faster than the BB1, DY and SDC methods, and competitive with the ABBmin2 method.
As for the other four problem sets, ANGM method with a small value of $\tau_1$ outperforms the other compared methods though its performance seems to become worse as $\tau_1$ increases.
The results shown in Tables \ref{tbANGR1} and \ref{tbANGR2} are slightly different from those in Table \ref{tbANGM}. In particular, ANGR1 and ANGR2 outperform other four compared methods
for most of the problem sets and values of $\tau_1$. For each given $\tau_1$ and tolerance level, ANGR1 and ANGR2 always perform better than other methods in terms of
total number of iterations.

\begin{table}[th!b]
\setlength{\tabcolsep}{0.25ex}
\caption{The numbers of averaged iterations of the ANGR2, BB1, DY, ABBmin2 and SDC methods on problems in Table \ref{tbspe}}\label{tbANGR2}
\centering
\begin{tiny}
\begin{tabular}{|c|c|c|c|c|c|c|c|c|c|c|c|c|c|c|}
\hline
 \multicolumn{1}{|c|}{\multirow{2}{*}{set}} &\multicolumn{1}{c|}{\multirow{2}{*}{$\epsilon$}}
 &\multicolumn{9}{c|}{$\tau_1$ for ANGR2} &\multirow{2}{*}{BB1} &\multirow{2}{*}{DY} &\multirow{2}{*}{ABBmin2} &\multirow{2}{*}{SDC}\\
\cline{3-11}
  \multicolumn{1}{|c|}{}& \multicolumn{1}{c|}{}  &0.1  &0.2  &0.3  &0.4  &0.5   &0.6   &0.7   &0.8  &0.9 & & & &\\
 \hline
\multicolumn{1}{|c|}{\multirow{3}{*}{1}}																																			
& 1e-06  &   196.9  &   191.8  &   184.8  &   186.8  &   183.3  &   189.0  &   \textbf{175.6}  &   182.2  &   189.5    &   221.0  &   197.1  &   177.1  &   183.2  \\
& 1e-09  &   591.0  &   607.8  &   599.4  &   613.3  &   669.5  &   699.6  &   704.0  &   763.6  &  1005.1    &  2590.0  &  2672.7  &   \textbf{428.8}  &  2029.6  \\
& 1e-12  &   870.4  &   811.3  &   763.1  &   790.7  &   800.0  &   881.4  &   844.8  &   971.0  &  1146.5    &  6032.9  &  6353.9  &   \textbf{560.3}  &  4087.2  \\
\hline

\multicolumn{1}{|c|}{\multirow{3}{*}{2}}
& 1e-06  &   129.7  &   117.6  &   115.0  &   \textbf{111.9}  &   114.6  &   116.3  &   140.6  &   146.9  &   164.4   &   311.2  &   261.5  &   302.5  &   160.6  \\
& 1e-09  &   \textbf{455.4}  &   474.8  &   470.8  &   491.4  &   530.6  &   566.0  &   595.9  &   601.8  &   691.5   &  1665.4  &  1340.0  &  1321.1  &   735.1  \\
& 1e-12  &   786.0  &   \textbf{750.5}  &   811.7  &   819.3  &   938.1  &   908.4  &  1000.3  &  1008.5  &  1089.0   &  2820.8  &  2434.6  &  2267.5  &  1346.8  \\
\hline

\multicolumn{1}{|c|}{\multirow{3}{*}{3}}
& 1e-06  &   156.2  &   141.2  &   \textbf{136.3}  &   149.7  &   146.7  &   167.9  &   179.9  &   188.0  &   206.8    &   388.4  &   329.4  &   356.3  &   235.5  \\
& 1e-09  &   495.8  &   \textbf{465.6}  &   490.4  &   529.0  &   530.8  &   606.4  &   667.6  &   718.5  &   726.6    &  1783.1  &  1511.8  &  1470.8  &   818.1  \\
& 1e-12  &   809.0  &   \textbf{775.1}  &   796.4  &   876.6  &   918.8  &   997.0  &  1080.8  &  1130.1  &  1157.8    &  2977.7  &  2780.2  &  2288.4  &  1310.6  \\
\hline

\multicolumn{1}{|c|}{\multirow{3}{*}{4}}
& 1e-06  &   183.2  &   168.0  &   \textbf{164.9}  &   166.9  &   180.9  &   195.7  &   219.8  &   234.2  &   225.3    &   500.5  &   431.3  &   519.0  &   262.8  \\
& 1e-09  &   \textbf{516.2}  &   521.9  &   523.2  &   541.6  &   603.9  &   637.3  &   700.0  &   709.8  &   739.2    &  1859.4  &  1659.9  &  1489.5  &   805.5  \\
& 1e-12  &   842.6  &   \textbf{830.1}  &   845.8  &   899.5  &   959.4  &  1039.2  &  1099.3  &  1133.9  &  1192.2    &  3051.5  &  2785.4  &  2383.9  &  1469.3  \\
\hline

\multicolumn{1}{|c|}{\multirow{3}{*}{5}}
& 1e-06  &   611.7  &   \textbf{580.5}  &   605.4  &   594.8  &   586.2  &   605.1  &   659.3  &   644.1  &   653.1    &   832.5  &   650.8  &   816.0  &   668.3  \\
& 1e-09  &  2472.7  &  \textbf{2394.4}  &  2510.8  &  2504.4  &  2447.2  &  2502.1  &  2551.9  &  2648.6  &  2763.9    &  4497.2  &  3185.5  &  2929.7  &  3274.5  \\
& 1e-12  &  4122.4  &  4108.7  &  \textbf{4103.2}  &  4107.2  &  4161.8  &  4227.9  &  4363.6  &  4324.0  &  4483.7    &  7446.7  &  7024.1  &  4808.7  &  5816.6  \\
\hline

\multicolumn{1}{|c|}{\multirow{3}{*}{total}}
& 1e-06   &  1277.8  &  \textbf{1199.1}  &  1206.3  &  1210.1  &  1211.6  &  1274.0  &  1375.2  &  1395.4  &  1439.1   &  2253.7  &  1870.1  &  2170.9  &  1510.4   \\
& 1e-09   &  4531.1  &  \textbf{4464.5}  &  4594.7  &  4679.7  &  4781.9  &  5011.5  &  5219.4  &  5442.3  &  5926.3   & 12395.0  & 10370.0  &  7640.0  &  7662.7   \\
& 1e-12   &  7430.3  &  \textbf{7275.7}  &  7320.2  &  7493.4  &  7778.1  &  8053.8  &  8388.8  &  8567.6  &  9069.2   & 22329.5  & 21378.3  & 12308.9  & 14030.4   \\
\hline

\end{tabular}
\end{tiny}
\end{table}

Secondly, we compared the methods on solving the non-rand quadratic problem \eqref{pro2} with $n=10,000$. For ANGM, ANGR1 and ANGR2, $\tau_1$ and $\tau_2$ were set to 0.4 and 1, respectively. The parameter pair $(h,s)$ used for the SDC method was set to $(30,2)$. Other settings are the same as above.
Table \ref{tbqp} presents averaged number of iterations over 10 different starting points with entries randomly generated in $[-10,10]$. It can be seen that
ANGM, ANGR1 and ANGR2  are significantly  better than the BB1 and DY methods.
In addition, ANGR1 and ANGR2 often outperform the ABBmin2 and SDC methods while ANGM is very competitive with them.

\begin{table}[th!b]
\setlength{\tabcolsep}{1.2ex}
\caption{The numbers of iterations of the compared methods on problem \eqref{pro2} with $n=10,000$}\label{tbqp}
\centering
\begin{footnotesize}
\begin{tabular}{|c|c|c|c|c|c|c|c|c|}
\hline
 \multicolumn{1}{|c|}{\multirow{1}{*}{$\kappa$}} &\multicolumn{1}{c|}{\multirow{1}{*}{$\epsilon$}}
  &\multirow{1}{*}{BB1} &\multirow{1}{*}{DY} &\multirow{1}{*}{ABBmin2} &\multirow{1}{*}{SDC} &\multicolumn{1}{c|}{ANGM}  &\multicolumn{1}{c|}{ANGR1} &\multicolumn{1}{c|}{ANGR2}\\
 \hline
\multicolumn{1}{|c|}{\multirow{3}{*}{$10^4$}}
& 1e-06  &   626.8  &   527.7  &   513.0  &   597.1  &   539.1		&   \textbf{500.6}		&   512.1\\
& 1e-09  &  1267.0  &   972.1  &   894.5  &  1000.6  &   976.7   &   893.7   &   \textbf{890.2}  \\
& 1e-12  &  1741.9  &  1396.8  &  1277.8  &  1409.2  &  1399.1   &  1298.0   &  \textbf{1257.4} \\
\hline

\multicolumn{1}{|c|}{\multirow{3}{*}{$10^5$}}
& 1e-06  &  1597.8  &  1326.7  &  1266.3  &  1254.3  &  1209.7		&  \textbf{1046.0}		&  1127.9\\
& 1e-09  &  3687.5  &  3168.3  &  2559.8  &  2647.4  &  2605.2   &  2424.3   &  \textbf{2399.8} \\
& 1e-12  &  5564.8  &  4892.4  &  3895.0  &  4156.4  &  4139.0   &  3858.5   &  \textbf{3663.3} \\
\hline

\multicolumn{1}{|c|}{\multirow{3}{*}{$10^6$}}
& 1e-06  &  4060.9  &  2159.4  &  3130.2  &  1986.2  &  2112.9		&  1992.0		&  \textbf{1936.0}\\
& 1e-09  & 10720.4  & 10134.3  &  7560.8  &  7178.5  &  7381.1   &  \textbf{6495.1}   &  6550.1 \\
& 1e-12  & 17805.5  & 18015.6  & 12193.6  & 11646.7  & 11922.9   & 10364.9   & \textbf{10280.2} \\
\hline

\multicolumn{1}{|c|}{\multirow{3}{*}{total}}
& 1e-06   &  6285.5  &  4013.8  &  4909.5  &  3837.6  &  3861.7		&  \textbf{3538.6}		&  3576.0 \\
& 1e-09   & 15674.9  & 14274.7  & 11015.1  & 10826.5  & 10963.0   &  \textbf{9813.1}   &  9840.1  \\
& 1e-12   & 25112.2  & 24304.8  & 17366.4  & 17212.3  & 17461.0   & 15521.4   & \textbf{15200.9}  \\
\hline
\end{tabular}
\end{footnotesize}
\end{table}

Finally, we compared the methods on solving two large-scale real problems Laplace1(a) and Laplace1(b) described in \cite{fletcher2005barzilai}. The two problems require the solution of a system of linear equations derived from a 3D Laplacian on a box, discretized using a standard 7-point finite difference stencil. Each problem has $n=N^3$ variables with $N$ being the number of interior nodes taken in each coordinate direction. The solution is fixed by a Gaussian function centered at $(\alpha,\beta,\gamma)$ and multiplied by $x(x-1)y(y-1)z(z-1)$. The parameter $\sigma$ is used to control the rate of decay of the Gaussian. See \cite{fletcher2005barzilai} for more details on these problems. Here, we set the parameters as follows:
\begin{align*}
&(a)~~ \sigma=20,~ \alpha=\beta=\gamma=0.5;\\
 &(b)~~ \sigma=50,~ \alpha=0.4,~ \beta=0.7,~ \gamma=0.5.
\end{align*}
The null vector was used as the starting point. We again stop the iteration when $\|g_k\|\leq\epsilon\|g_0\|$ with different values of $\epsilon$.


For ANGM, ANGR1 and ANGR2, $\tau_1$ and $\tau_2$ were set to 0.7 and 1.2, respectively. The parameter pair $(h,s)$ used for the SDC method was chosen for the best performance in our test, i.e., $(2,6)$ and $(8,6)$ for Laplace1(a) and Laplace1(b), respectively. Other settings are the same as above. The number of iterations required by the compared methods for solving the two problems are listed in Table \ref{tbnLap}. It can be seen that our methods are significantly better than the BB1, DY and SDC methods and
is often faster than ABBmin2 especially when a tight tolerance is used.

\begin{table}[th!b]
\caption{The numbers of iterations of BB1, DY, ABBmin2, SDC, ANGM, ANGR1 and ANGR2 on solving the 3D Laplacian problem}\label{tbnLap}
\centering
\begin{footnotesize}
\begin{tabular}{|c|c|c|c|c|c|c|c|c|}
\hline
 \multicolumn{1}{|c|}{\multirow{1}{*}{$n$}} &\multicolumn{1}{c|}{\multirow{1}{*}{$\epsilon$}}
  &\multirow{1}{*}{BB1} &\multirow{1}{*}{DY} &\multirow{1}{*}{ABBmin2} &\multirow{1}{*}{SDC} &\multicolumn{1}{c|}{ANGM}  &\multicolumn{1}{c|}{ANGR1} &\multicolumn{1}{c|}{ANGR2}\\
 \hline
 \multicolumn{9}{|c|}{Laplace1(a)}\\
  \hline
\multicolumn{1}{|c|}{\multirow{3}{*}{$60^{3}$}}
& 1e-06  &   259  &   249  &   \textbf{192}  &   213  &   245  &   195  &   233  \\
& 1e-09  &   441  &   373  &   329  &   393  &   313  &   322  &   \textbf{308}  \\
& 1e-12  &   680  &   546  &   401  &   529  &   367  &   373  &   \textbf{364}  \\
\hline

\multicolumn{1}{|c|}{\multirow{3}{*}{$80^{3}$}}
& 1e-06  &   359  &   383  &   289  &   297  &   291  &   332  &   \textbf{288}  \\
& 1e-09  &   591  &   570  &   430  &   553  &   408  &   446  &   \textbf{396}  \\
& 1e-12  &   882  &   789  &   608  &   705  &   620  &   \textbf{516}  &   591  \\
\hline

\multicolumn{1}{|c|}{\multirow{3}{*}{$100^{3}$}}
& 1e-06  &   950  &   427  &   351  &   513  &   450  &   \textbf{303}  &   416  \\
& 1e-09  &  1088  &   651  &   485  &   609  &   584  &   519  &   \textbf{503}  \\
& 1e-12  &  1241  &   918  &   687  &   825  &   694  &   604  &   \textbf{597}  \\
\hline

\multicolumn{1}{|c|}{\multirow{3}{*}{total}}
& 1e-06   &  1568  &  1059  &   832  &  1023  &   986  &   \textbf{830}  &   937 \\
& 1e-09   &  2120  &  1594  &  1244  &  1555  &  1305  &  1287  &  \textbf{1207} \\
& 1e-12   &  2803  &  2253  &  1696  &  2059  &  1681  &  \textbf{1493}  &  1552 \\
\hline

\multicolumn{9}{|c|}{Laplace1(b)}\\
 \hline

 \multicolumn{1}{|c|}{\multirow{3}{*}{$60^{3}$}}
& 1e-06  &   246  &   236  &   217  &   \textbf{213}  &   242  &   217  &   214  \\
& 1e-09  &   473  &   399  &   365  &   437  &   \textbf{333}  &   338  &   409  \\
& 1e-12  &   651  &   532  &   502  &   555  &   \textbf{451}  &   478  &   573  \\
\hline

\multicolumn{1}{|c|}{\multirow{3}{*}{$80^{3}$}}
& 1e-06  &   288  &   454  &   294  &   309  &   296  &   \textbf{290}  &   324  \\
& 1e-09  &   607  &   567  &   \textbf{433}  &   485  &   517  &   499  &   495  \\
& 1e-12  &   739  &   794  &   634  &   766  &   686  &   \textbf{590}  &   645  \\
\hline

\multicolumn{1}{|c|}{\multirow{3}{*}{$100^{3}$}}
& 1e-06  &   544  &   371  &   369  &   379  &   381  &   406  &   \textbf{358}  \\
& 1e-09  &   646  &   700  &   585  &   653  &   638  &   \textbf{558}  &   648  \\
& 1e-12  &   937  &  1038  &   880  &   965  &   854  &   \textbf{785}  &   810  \\
\hline

\multicolumn{1}{|c|}{\multirow{3}{*}{total}}
& 1e-06   &  1078  &  1061  &   \textbf{880}  &   901  &   919  &   913  &   896 \\
& 1e-09   &  1726  &  1666  &  \textbf{1383}  &  1575  &  1488  &  1395  &  1552 \\
& 1e-12   &  2327  &  2364  &  2016  &  2286  &  1991  &  \textbf{1853}  &  2028 \\
\hline
\end{tabular}
\end{footnotesize}
\end{table}

\subsection{Unconstrained problems}
Now we present comparisons of ANGR1 and ANGR2 with SPG method\footnote{codes available at \url{https://www.ime.usp.br/~egbirgin/tango/codes.php}} in \cite{birgin2000nonmonotone,birgin2014spectral}, and the BB1 method using Dai-Zhang nonmonotone line search \cite{dai2001adaptive} (BB1-DZ)
on solving general unconstrained problems.

For our methods, the parameter values are set as the following:
\begin{equation*}
  \alpha_{\min}=10^{-30},~\alpha_{\max}=10^{30},~M=8,~\sigma=10^{-4},~\alpha_0=1/\|g_0\|_\infty.
\end{equation*}
In addition, the parameter $\tau_1$ and $\tau_2$ for ANGR1 and ANGR2 are set to 0.8 and 1.2, respectively. Default parameters were used for SPG. Each method was stopped if the number of iteration exceeds 200,000 or $\|g_k\|_\infty\leq10^{-6}$.

Our test problems were taken from \cite{andrei2008unconstrained}.
 We have tested 59 problems listed in Table \ref{tbunpro} with $n=1,000$
 and the performance profiles are shown in  Fig. \ref{Andrei},
 which  shows that ANGR1 and ANGR2 outperform SPG and BB1-DZ in terms of the iteration number,
 and ANGR2 is faster than ANGR1.
Moreover, BB1-DZ is slightly better than SPG. Detail numerical results are also presented in Table \ref{tbunc}.
Since the only difference between BB1-DZ with ANGR1 and ANGR2 lies in the choice of stepsizes,
these numerical results show our adaptive choices of stepsizes in ANGR1 and ANGR2 are very effective
 and can indeed greatly accelerate the convergence of BB-type methods.

\begin{table}[ht!b]
\setlength{\tabcolsep}{1.5ex}
\caption{Test problems from \cite{andrei2008unconstrained}}\label{tbunpro}
\centering
\begin{tiny}
\begin{tabular}{|c|c|c|c|}
\hline
Problem & name &Problem & name\\
\hline
  1      &Extended Freudenstein \& Roth function 		&31      &                        NONDIA\\
  2      &        Extended Trigonometric            &32      &                       DQDRTIC\\
  3      &        Extended White \& Holst           &33      &   Partial Perturbed Quadratic\\
  4      &                Extended Beale              &34      &           Broyden Tridiagonal\\
  5      &              Extended Penalty              &35      &    Almost Perturbed Quadratic\\
  6      &           Perturbed Quadratic              &36      &Perturbed Tridiagonal Quadratic\\
  7      &                      Raydan 1              &37      &                   Staircase 1\\
  8      &                      Raydan 2              &38      &                   Staircase 2\\
  9      &                    Diagonal 1              &39      &                       LIARWHD\\
  10      &                    Diagonal 2             &40      &                         POWER\\
  11      &                    Diagonal 3             &41      &                       ENGVAL1\\
  12      &                         Hager             &42      &                       EDENSCH\\
  13      &     Generalized Tridiagonal 1             &43      &                         BDEXP\\
  14      &        Extended Tridiagonal 1             &44      &                      GENHUMPS\\
  15      &                  Extended TET             &45      &                      NONSCOMP\\
  16      &     Generalized Tridiagonal 2             &46      &                        VARDIM\\
  17      &                    Diagonal 4             &47      &                        QUARTC\\
  18      &                    Diagonal 5             &48      &             Extended DENSCHNB\\
  19      &           Extended Himmelblau             &49      &             Extended DENSCHNF\\
  20      &                 Extended PSC1             &50      &                       LIARWHD\\
  21      &              Generalized PSC1             &51      &                       BIGGSB1\\
  22      &               Extended Powell             &52      &           Generalized Quartic\\
  23      &                Extended Cliff             &53      &                    Diagonal 7\\
  24      &  Perturbed quadratic diagonal             &54      &                    Diagonal 8\\
  25      &                 Quadratic QF1             &55      &              Full Hessian FH3\\
  26      &Extended quadratic exponential EP1        &56      &                        SINCOS\\
  27      &        Extended Tridiagonal 2             &57      &                    Diagonal 9\\
  28      &                       BDQRTIC             &58      &                      HIMMELBG\\
  29      &                        TRIDIA             &59      &                       HIMMELH\\
  30      &                       ARWHEAD  						&					&\\

\hline
\end{tabular}
\end{tiny}
\end{table}

\begin{figure}[th!b]
\centering
    \includegraphics[width=0.55\textwidth,height=0.4\textwidth]{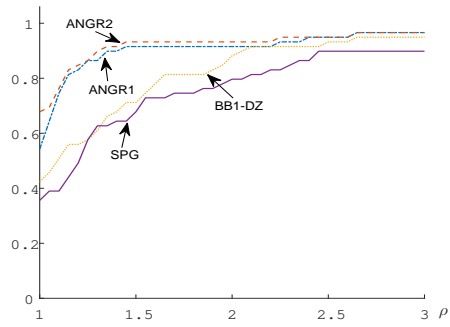}
  \caption{Performance profiles of compared methods on 59 unconstrained problems in Table \ref{tbunpro}, iteration metric}\label{Andrei}
\end{figure}

\begin{table}[th!b]
\setlength{\tabcolsep}{0.15ex}
\caption{Results of compared methods on unconstrained problems in Table \ref{tbunpro}}\label{tbunc}
\centering
\begin{tiny}
\begin{tabular}{|c|c|c|c|c|c|c|c|c|c|c|c|c|c|c|c|c|}
\hline
 \multicolumn{1}{|c|}{\multirow{2}{*}{Problem}} &\multicolumn{3}{c|}{\multirow{1}{*}{SPG}}
 &\multicolumn{3}{c|}{\multirow{1}{*}{BB1-DZ}}
 &\multicolumn{3}{c|}{\multirow{1}{*}{ANGR1}}
&\multicolumn{3}{c|}{\multirow{1}{*}{ANGR2}}   \\
\cline{2-13}
  \multicolumn{1}{|c|}{} & iter  &$f_k$  &$\|g_k\|_\infty$   & iter  &$f_k$  &$\|g_k\|_\infty$   & iter  &$f_k$  &\multicolumn{1}{c|}{$\|g_k\|_\infty$}
  & iter  &$f_k$  &\multicolumn{1}{c|}{$\|g_k\|_\infty$}\\
\hline

  1 &    87  &2.45e+04  &8.95e-07   &    41  &2.45e+04  &4.06e-08   &    26  &2.45e+04  &1.46e-07   &    25  &2.45e+04  &1.46e-07  \\
 2 &    45  &5.72e-13  &3.83e-07   &    87  &5.26e-07  &3.97e-07   &   100  &4.89e-07  &9.32e-07   &   100  &4.73e-07  &2.20e-07  \\
 3 &   110  &5.80e-14  &3.05e-08   &   120  &6.30e-20  &4.95e-11   &    91  &9.95e-12  &2.93e-07   &    88  &1.05e-11  &2.90e-07  \\
 4 &    46  &3.92e-10  &6.57e-07   &    65  &8.28e-18  &3.74e-10   &    31  &9.52e-15  &4.15e-08   &    31  &1.94e-12  &6.12e-08  \\
 5 &    41  &8.83e+02  &9.02e-07   &   107  &8.83e+02  &1.24e-08   &   107  &8.83e+02  &1.24e-08   &   107  &8.83e+02  &1.24e-08  \\
 6 &   597  &2.39e-13  &9.83e-07   &   457  &2.16e-13  &9.24e-07   &   300  &7.98e-16  &5.66e-08   &   287  &1.69e-13  &8.03e-07  \\
 7 &   465  &5.01e+04  &9.64e-07   &   339  &5.01e+04  &9.04e-07   &   330  &5.01e+04  &5.56e-08   &   301  &5.01e+04  &8.39e-07  \\
 8 &     1  &1.00e+03  &0.00e+00   &     1  &1.00e+03  &0.00e+00   &     1  &1.00e+03  &0.00e+00   &     1  &1.00e+03  &0.00e+00  \\
 9 &   415  &-2.71e+06  &8.50e-07   &   361  &-2.71e+06  &5.43e-07   &   294  &-2.71e+06  &9.05e-07   &   276  &-2.71e+06  &7.04e-07  \\
 10 &   168  &3.13e+01  &6.05e-07   &   173  &3.13e+01  &5.68e-07   &   135  &3.13e+01  &5.93e-07   &   131  &3.13e+01  &8.61e-07  \\
 11 &   668  &-4.96e+05  &8.82e-07   &   404  &-4.96e+05  &8.63e-07   &   336  &-4.96e+05  &8.85e-07   &   286  &-4.96e+05  &9.93e-07  \\
 12 &    63  &-4.47e+04  &2.91e-07   &    58  &-4.47e+04  &5.70e-07   &    52  &-4.47e+04  &5.14e-07   &    52  &-4.47e+04  &5.59e-07  \\
 13 &    29  &9.97e+02  &6.53e-07   &    27  &9.97e+02  &9.67e-07   &    26  &9.97e+02  &4.21e-07   &    25  &9.97e+02  &9.79e-07  \\
 14 &    30  &2.13e-07  &3.75e-07   &    18  &2.68e-07  &4.46e-07   &    20  &5.20e-07  &7.33e-07   &    20  &5.20e-07  &7.33e-07  \\
 15 &    10  &1.28e+03  &4.34e-10   &     5  &1.28e+03  &1.33e-08   &     5  &1.28e+03  &1.33e-08   &     5  &1.28e+03  &1.33e-08  \\
 16 &    69  &2.38e+00  &8.58e-07   &    58  &9.58e-01  &7.14e-07   &    56  &9.58e-01  &8.40e-07   &    57  &9.58e-01  &6.13e-07  \\
 17 &     3  &0.00e+00  &0.00e+00   &     3  &0.00e+00  &0.00e+00   &     3  &0.00e+00  &0.00e+00   &     3  &0.00e+00  &0.00e+00  \\
 18 &     4  &6.93e+02  &4.81e-08   &     1  &6.93e+02  &0.00e+00   &     1  &6.93e+02  &0.00e+00   &     1  &6.93e+02  &0.00e+00  \\
 19 &    15  &4.21e-19  &1.92e-10   &    15  &4.21e-19  &1.92e-10   &    15  &4.21e-19  &1.92e-10   &    15  &4.21e-19  &1.92e-10  \\
 20 &    23  &9.99e+02  &8.27e-07   &    22  &9.99e+02  &8.65e-07   &    22  &9.99e+02  &8.65e-07   &    22  &9.99e+02  &8.65e-07  \\
 21 &    15  &3.87e+02  &5.87e-07   &    13  &3.87e+02  &3.78e-07   &    13  &3.87e+02  &3.78e-07   &    13  &3.87e+02  &3.78e-07  \\
 22 &   246  &3.19e-07  &6.80e-07   &   266  &3.67e-07  &5.97e-07   &   164  &7.08e-07  &9.62e-07   &   205  &5.55e-07  &9.82e-07  \\
 23 &   154  &9.99e+01  &7.79e-07   &   126  &9.99e+01  &1.57e-07   &   541  &9.99e+01  &2.61e-07   &   539  &9.99e+01  &6.10e-08  \\
 24 &   373  &2.41e-11  &8.95e-07   &   246  &1.41e-11  &6.83e-07   &   249  &2.05e-12  &2.98e-07   &   314  &4.05e-12  &4.14e-07  \\
 25 &   612  &-5.00e-04  &9.45e-07   &   575  &-5.00e-04  &9.95e-07   &   319  &-5.00e-04  &8.92e-08   &   282  &-5.00e-04  &7.45e-07  \\
 26 &     4  &7.93e+03  &4.19e-08   &     4  &7.93e+03  &3.71e-09   &     4  &7.93e+03  &3.71e-09   &     4  &7.93e+03  &3.71e-09  \\
 27 &    34  &3.89e+02  &7.50e-07   &    39  &3.89e+02  &9.74e-07   &    36  &3.89e+02  &8.49e-07   &    36  &3.89e+02  &5.15e-07  \\
 28 &    77  &3.98e+03  &7.64e-07   &    98  &3.98e+03  &8.91e-07   &    69  &3.98e+03  &1.49e-07   &    64  &3.98e+03  &9.57e-07  \\
 29 &  3910  &2.47e-13  &6.97e-07   &  2403  &9.93e-15  &6.71e-07   &   816  &9.62e-16  &4.67e-07   &   611  &6.71e-14  &5.86e-07  \\
 30 &     4  &0.00e+00  &1.50e-09   &     4  &0.00e+00  &1.50e-09   &     4  &0.00e+00  &1.50e-09   &     4  &0.00e+00  &1.50e-09  \\
 31 &    15  &1.35e-14  &9.86e-09   &    15  &1.35e-14  &9.86e-09   &    16  &1.10e-12  &8.20e-08   &    16  &1.10e-12  &8.20e-08  \\
 32 &    41  &4.20e-16  &2.86e-07   &    26  &1.80e-13  &8.47e-07   &    17  &2.99e-16  &2.14e-07   &    17  &2.99e-16  &2.14e-07  \\
 33 &   312  &2.51e-14  &5.96e-07   &   229  &1.42e-13  &9.63e-07   &   257  &6.00e-15  &2.16e-07   &   294  &7.03e-14  &8.24e-07  \\
 34 &    47  &5.31e-15  &3.20e-07   &    51  &7.94e-14  &8.28e-07   &    41  &7.67e-14  &7.35e-07   &    47  &2.62e-14  &7.66e-07  \\
 35 &   621  &2.36e-13  &9.76e-07   &   363  &2.47e-13  &9.97e-07   &   358  &1.53e-13  &7.77e-07   &   272  &4.53e-15  &6.22e-07  \\
 36 &   595  &2.49e-13  &9.97e-07   &   606  &2.48e-13  &9.95e-07   &   368  &3.46e-13  &1.00e-06   &   306  &1.82e-13  &8.61e-07  \\
 37 & 15936  &1.81e-12  &7.80e-07   & 10239  &1.54e-12  &7.68e-07   &  7234  &1.39e-13  &1.89e-07   &  6543  &1.07e-12  &9.13e-07  \\
 38 & 28528  &1.97e-14  &2.43e-07   & 25640  &2.74e-12  &9.97e-07   & 60283  &2.86e-12  &9.93e-07   & 27783  &2.98e-12  &9.74e-07  \\
 39 &    61  &2.64e-14  &2.06e-08   &    59  &2.62e-13  &4.93e-07   &    47  &9.15e-14  &6.15e-08   &    47  &1.30e-12  &2.36e-07  \\
 40 &  5435  &2.84e-13  &9.98e-07   &  5472  &2.94e-13  &9.98e-07   &   327  &1.89e-13  &6.99e-07   &   366  &5.63e-13  &9.98e-07  \\
 41 &    33  &1.11e+03  &8.58e-07   &    31  &1.11e+03  &8.09e-07   &    29  &1.11e+03  &1.70e-07   &    29  &1.11e+03  &1.70e-07  \\
 42 &    32  &6.00e+03  &7.33e-07   &    32  &6.00e+03  &7.33e-07   &    31  &6.00e+03  &4.46e-07   &    31  &6.00e+03  &9.35e-07  \\
 43 &    15  &6.52e-05  &9.93e-07   &    15  &6.52e-05  &9.93e-07   &    15  &6.52e-05  &9.93e-07   &    15  &6.52e-05  &9.93e-07  \\
 44 &   747  &1.39e-17  &1.50e-09   &     3  &4.49e-23  &3.00e-12   &     3  &4.49e-23  &3.00e-12   &     3  &4.49e-23  &3.00e-12  \\
 45 &    48  &9.98e-12  &8.75e-07   &    89  &1.97e-14  &2.57e-07   &    59  &2.30e-13  &5.85e-07   &    63  &2.18e-13  &7.89e-07  \\
 46 &     1  &4.62e-30  &4.44e-16   &    30  &2.17e-27  &5.77e-15   &    30  &2.17e-27  &5.77e-15   &    30  &2.17e-27  &5.77e-15  \\
 47 &     1  &0.00e+00  &0.00e+00   &     1  &0.00e+00  &0.00e+00   &     1  &0.00e+00  &0.00e+00   &     1  &0.00e+00  &0.00e+00  \\
 48 &     8  &1.37e-11  &4.68e-07   &     8  &1.37e-11  &4.68e-07   &     8  &1.37e-11  &4.68e-07   &     8  &1.37e-11  &4.68e-07  \\
 49 &    12  &1.78e-19  &4.76e-10   &    17  &6.46e-21  &5.91e-11   &    17  &6.46e-21  &5.91e-11   &    17  &6.46e-21  &5.91e-11  \\
 50 &    61  &2.64e-14  &2.06e-08   &    59  &2.62e-13  &4.93e-07   &    47  &9.15e-14  &6.15e-08   &    47  &1.30e-12  &2.36e-07  \\
 51 & 30459  &1.24e-05  &9.95e-07   & 14877  &3.73e-07  &6.46e-07   &  7558  &6.51e-07  &8.62e-07   &  8969  &6.08e-12  &6.45e-07  \\
 52 &     9  &4.98e-15  &1.99e-07   &     9  &4.98e-15  &1.99e-07   &     9  &4.98e-15  &1.99e-07   &     9  &4.98e-15  &1.99e-07  \\
 53 &     7  &-8.17e+02  &1.97e-09   &     7  &-8.17e+02  &1.97e-09   &     7  &-8.17e+02  &1.97e-09   &     7  &-8.17e+02  &1.97e-09  \\
 54 &     6  &-4.80e+02  &5.85e-08   &     6  &-4.80e+02  &4.36e-10   &     6  &-4.80e+02  &4.36e-10   &     6  &-4.80e+02  &4.36e-10  \\
 55 &     3  &-2.50e-01  &4.56e-10   &     3  &-2.50e-01  &4.56e-10   &     3  &-2.50e-01  &4.56e-10   &     3  &-2.50e-01  &4.56e-10  \\
 56 &    15  &3.87e+02  &5.87e-07   &    13  &3.87e+02  &3.78e-07   &    13  &3.87e+02  &3.78e-07   &    13  &3.87e+02  &3.78e-07  \\
 57 &   553  &-2.70e+06  &1.91e-07   &   762  &-2.70e+06  &6.06e-07   &   305  &-2.70e+06  &9.96e-07   &   346  &-2.70e+06  &5.60e-07  \\
 58 &    16  &4.79e-04  &8.77e-07   &    17  &3.78e-04  &6.91e-07   &    17  &3.78e-04  &6.91e-07   &    17  &3.78e-04  &6.91e-07  \\
 59 &    12  &-5.00e+02  &1.09e-07   &    12  &-5.00e+02  &7.41e-11   &    12  &-5.00e+02  &7.41e-11   &    12  &-5.00e+02  &7.41e-11  \\

\hline
\end{tabular}
\end{tiny}
\end{table}

\subsection{Bound constrained problems}
We further compare ANGR1 and ANGR2, with SPG \cite{birgin2000nonmonotone,birgin2014spectral} and the BB1-DZ method \cite{dai2001adaptive}
combined with gradient projection techniques on solving bound constrained problems from the CUTEst collection  \cite{gould2015cutest} with dimension more than $50$.
We deleted $3$ problems from this list since none of these comparison algorithms can solve them. Hence, in total there are $47$ problems left in our test.
The iteration was stopped if the number of iteration exceeds 200,000 or  $\|P_{\Omega}(x_k-g_k)-x_k\|_{\infty}\leq10^{-6}$.
The parameters $\tau_1$ and $\tau_2$ for ANGR1 and ANGR2 are set to 0.4 and 1.5, respectively.
Other settings are the same as before.
Fig. \ref{cutest} shows the performance profiles of all the compared methods on iteration metric.
Similar as the unconstrained case, from  Fig. \ref{cutest}, we again see that both ANGR1 and ANGR2 perform significantly better than SPG and BB1-DZ. Hence, our new gradient methods also have potential great benefits
for solving constrained optimization.

\begin{figure}[thp!b]
  \centering
  \includegraphics[width=0.55\textwidth,height=0.4\textwidth]{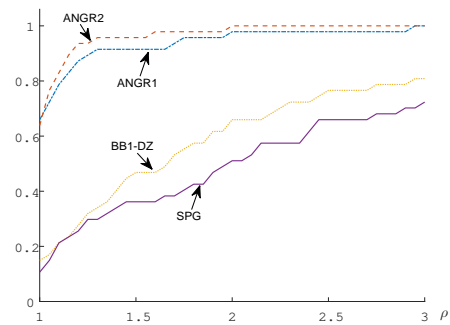}\\
  \caption{Performance profiles of compared methods on solving 47 bound constrained problems from CUTEst, iteration metric}\label{cutest}
\end{figure}

\section{Conclusions}\label{secclu}
We have developed techniques to accelerate the Barzilai-Borwein (BB) method motivated from finite termination for minimizing two-dimensional strongly convex quadratic
functions. More particularly, by exploiting certain orthogonal properties of the gradients, we derive a new monotone stepsize that can be combined with BB stepsizes to
significantly improve their performance for minimizing general strongly convex quadratic functions.
By adaptively using this new stepsize and the two BB stepsizes, we develop a new gradient method called ANGM and its two variants ANGR1 and ANGR2,
which are further extended for solving unconstrained and bound constrained optimization.
Our extensive numerical experiments show that all the new developed methods are significantly better
than the BB method
and are faster than some very successful gradient methods developed in the recent literature for solving quadratic, general smooth unconstrained
and bound constrained  optimization.


\begin{acknowledgements}
This work was supported by the National Natural Science Foundation of China (Grant Nos. 11631013, 11701137, 11671116), the National 973 Program of China (Grant No. 2015CB856002),  China Scholarship Council (No. 201806705007), and USA National Science Foundation (1522654, 1819161).
\end{acknowledgements}


\end{document}